%% file: arxivOct26.tex
\documentclass[12pt, reqno]{amsart}
\input{preamble}

\usepackage{wasysym}
\usepackage{mathtools}
\usepackage{breqn}
\usepackage{times}
\newcommand\term[1]{\emph{#1}} %\marginnote{\textbf{#1}}

\def\1{\mathbf{1}}

\def\Z{\mathbb Z}
\def\E{\mathbb E}

\author{Francois Baccelli}
\address[Francois Baccelli]{Department of Mathematics, University of Texas at Austin, Austin TX 78712, USA}
\email{baccelli@math.utexas.edu} 

\author{Ngoc Mai Tran}
\address[Ngoc Mai Tran]{Department of Mathematics, University of Texas at Austin, Austin TX 78712, USA; Department of Mathematics, University of Bonn, Bonn 53113, Germany; and the Hausdorff Center for Mathematics, Bonn 53113, Germany}
\email{ntran@math.utexas.edu}

\def\G{\mathcal{G}}

\title{Iterated Gilbert Mosaics \\ and Poisson Tropical Plane Curves}

\begin{document}
%\pdfgraphics
\begin{abstract}
We propose an iterated version of the Gilbert model, which results in a sequence of random mosaics of the plane. We prove that under appropriate scaling, this sequence of mosaics converges to that obtained by a classical Poisson line process with explicit cylindrical measure. Our model arises from considerations on tropical plane curves, which are zeros of random tropical polynomials in two variables. In particular, the iterated Gilbert model convergence allows one to derive a scaling limit for Poisson tropical plane curves. Our work raises a number of open questions at the intersection of stochastic and tropical geometry.
\end{abstract}
\maketitle

\section{Introduction}
The Gilbert model is a random mosaic of the plane obtained by letting line segments (cracks) grow from a homogeneous Poisson point process at constant speed, with the rule that a crack stops growing the instance it hits another crack. In the computer vision literature, this is known as the {\em motorcycle graph}, in reference to the 1982 Disney movie Tron \cite{eppstein2008motorcycle}. Efficient computations of the motorcycle graph starting from a fixed set of points and directions are important for generating quadrilateral meshes in computer graphics. A closely related model is the lilypond model, where the entire growth, rather than just the directional growth, is blocked upon collision with another object \cite{haggstrom1996nearest}. The lilypond has attracted much attention in stochastic geometry and percolation theory \cite{cotar2004note,heveling2006existence,last2013percolation,daley2005descending,cotar2009percolating,daleyf2004further}.

Geometric functionals of the Gilbert model are notoriously difficult to obtain exactly. The expected length of a typical line segment, for instance, is only known in a few special cases \cite{burridge2013full}.
Another approach is to look for fluctuations in a large window of functionals of the summation kind.
As the window size increases, one may expect law of large numbers and central limit theorems to hold. Schreiber and Soja \cite{schreiber2010limit} proved such results for a large class of geometric functionals using stabilization theory. This is, to our knowledge, one of the few scaling limit results for the Gilbert model. 

In this work, we iterate the Gilbert construction to obtain a family of random mosaics $\mathcal{G}^k$ for $k \in \mathbb{N}$, where the classical Gilbert mosaic is the case $k = 1$. Roughly speaking, we allow a line to collide with $k$ other lines before it stops growing. As $k$ increases, the intensity of intersections also increases, leading to shorter line segments and smaller facets. Our main result, Theorem~\ref{thm:main}, states that appropriately scaled, this sequence of random mosaics converges in the vague topology to a classical Poisson line process with explicit law. 

The second part of our paper presents an application of the iterated Gilbert model to tropical geometry. We use Theorem \ref{thm:main} to obtain a scaling limit for random tropical plane curves, as well as the asymptotic growth rates of various functionals.
Tropical geometry is the study of tropical varieties,
which are are limits of classical algebraic varieties
under the logarithm map \cite{maclagan2015introduction}.
They are also zeros of polynomials in the tropical (min-plus) semi-ring.
As sets, they are piecewise-linear, formed by intersections of
affine hyperplanes. Thus, they form natural intermediates between
classical algebraic varieties and affine structures,
and provide an attractive extensions to classical affine models
studied in stochastic geometry.

Our setting here can be seen as a continuation
of the work \cite{baccelli2016zeros}, where we studied 
the asymptotic number of zeros
for one random tropical polynomial in one variable, of degree $n$,
and with i.i.d. coefficients, when $n$ tends to infinity.
The iterated Gilbert model presented here is motivated by the analysis
of the {\em common zeros of an infinite and scale
invariant random system of tropical polynomials
in two variables}, obtained from an i.i.d. sequence of polynomials
and a Poisson point process. 

The paper is organized as follows. In Section \ref{sec:background}
we define the iterated Gilbert mosaic and prove some of its basic properties.
In Section \ref{sec:main.proof}, we state and prove the main result,
Theorem~\ref{thm:main}, and its generalization, Theorem \ref{thm:main2}.
The heart of the proof is an induction argument, broken up into a series
of lemmas. In Section \ref{sec:tropical}, we gather basic facts from
tropical geometry, and motivate the definition
of the translation invariant process of tropical plane curves associated
with the polynomial ensemble alluded to above.
We apply Theorem \ref{thm:main2} to obtain a scaling limit for
the tropical plane curves process in Theorem \ref{thm:main.tropical}.
We then focus on the case of the tropical line process,
giving various statistics such as the relative densities of
tropical polytopes of various types.
Section \ref{sec:discussions} concludes with open problems
of interest to both stochastic and tropical algebraic geometers.

\subsection*{Notation}
For a subset $W \subset \R^2$, let $|W|$ denote its area under Lebesgue measure. For a finite set $Q$, let $|Q|$ denote its cardinality. Write $\vect{\phi}$ for the unit vector in direction $\phi$. Let $\mathbf{1}$ be the all-one vector.

\section{The iterated Gilbert mosaic}\label{sec:background}

\begin{defn}[Iterated Gilbert model]\label{defn:gk}
Let $\mathcal{P}$ be a compound Poisson point process $\R^2$ with intensity $\lambda$ and multiplicity measure $\mathcal{M}$ supported on a subset of $\{1,2,\ldots,M\}$ for $3 \leq M <~\infty$. Let $A$ be a set of $M$ angles in $[0,2\pi)$. Write $\A = (\A^1, \A^2, \ldots, \A^M)$, where $\A^m$ is a distribution on the product set $A^m$ such that no two coordinates are equal. For $k \geq 1$,
 the \emph{$k$-th order Gilbert model} $\mathcal{G}^k(\mathcal{P}, \A)$ is the random closed set (RACS) resulting from the following construction (the fact that $\G^{k}({\mathcal P},\A)$ is a RACS is proven below). At time $t = 0$, independently at each site $p \in \mathcal{P}$ with multiplicity $m(p) = m \in \{1,2,\ldots, M\}$, pick $m$ directions
$\phi_1, \ldots, \phi_m$ jointly according to the distribution $\A^m$. Put $m$ motorcycles at $p$, one for each travel direction. As time $t$ increases, each motorcycle then travels at velocity $1$ in its prescribed direction, leaving behind a poisonous line. Each motorcycle initially has $k$ lives. At time $t > 0$, if a motorcycle $b$ touches the line of another motorcycle $b'$, it loses one life. The instance the motorcycle has zero lives, it vanishes. Let $\G^{k,t}(\mathcal P,\A)$ denote the union of the lines that have appeared up until time $t$. If almost surely, for each compact window $W \subset \R^2$, $\G^{k,t}(\mathcal P,\A) \cap W$ is equal to a fixed set after finite time, define 
$$ \G^k(\mathcal P,\A) := \lim_{t\to\infty}\G^{k,t}(\mathcal P,\A), $$
in the sense that for each compact $W \subset \R^2$, 
$$ \G^k(\mathcal P,\A) \cap W = \lim_{t\to\infty}\G^{k,t}(\mathcal P,\A) \cap W. $$
\end{defn}

For $m \in \{1,2,\ldots,M\}$, let $\pi_m$ be the probability that a site $p$ in $\mathcal{P}$ has multiplicity $m(p) = m$. One can view $\mathcal{P}$ as the superposition of $2^M-1$ independent Poisson point processes $\{\mathcal{P}_Q, Q \subseteq A\}$  on $\R^2$, where $\mathcal{P}_Q$ is the set of sites with angles $Q$, which is a Poisson point process with intensity
$$\mu _Q := \lambda \pi_{|Q|}\mathcal{A}^{|Q|}(Q).$$

We shall refer to a motorcycle as a marked point $\bar{b} := (b,\phi)$, consisting of its origin $b \in \R^2$ and travel direction $\phi \in [0,2\pi)$. Write $\overline{\mathcal{P}}$ for the set of marked points generated at the beginning of time. If $p \in \R^2$ is a point on the path of $\bar{b}$, define the \term{age} of $\bar{b}$ at $p$ to be the time at which it reaches $p$. When $\bar{b}$ loses one life due to another motorcycle $\bar{b}'$ at location $p \in \R^2$, we say that \term{$\bar{b}'$ kills $\bar{b}$ at $p$}, or that \emph{$\bar{b}'$ is a killer of $\bar{b}$ at $p$}. The location where $\bar{b}$ vanishes is called its \emph{grave}. The classical Gilbert corresponds to $k = 1$. When no confusion can arise, we write $\G^k$ for $\G^k(\mathcal P,\A)$.

\begin{ex}[Non-monotonicity of the iterated Gilbert sets]\label{ex:iterated.gilbert}
It is important to note that the sequence of random closed sets $\{\G^k(\mathbb{P}, \mathcal{A}): k \geq 1\}$ may \emph{not } be a.s. monotone increasing with $k$, as shown in the example of Figure \ref{fig1}.
\begin{figure}[h]
\includegraphics[width=\textwidth]{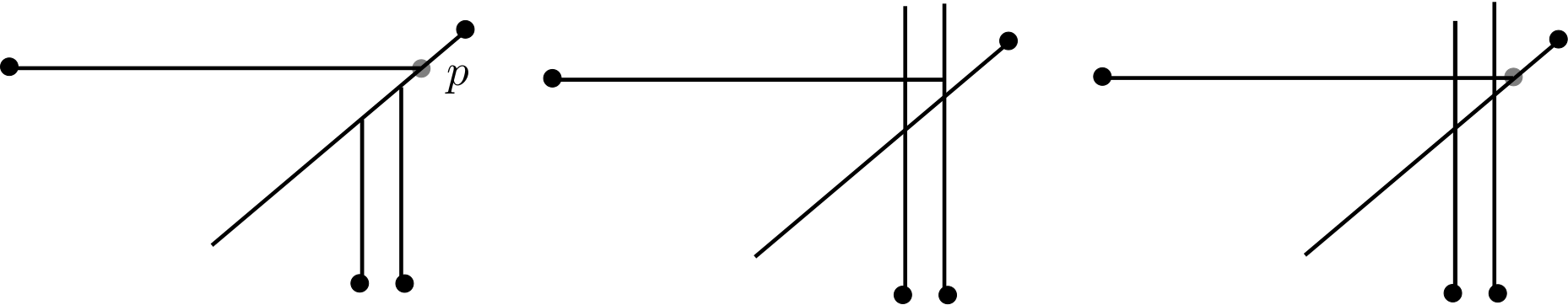}
\caption{Black points are points of $\mathcal{P}$. From left to right: an iterated Gilbert process of order $k = 1,2$ and $3$, respectively. Note that $p \in \G^1$ and $p \in \G^3$ but $p \notin \G^2$.}
\label{fig1}
\end{figure}
\end{ex}

Below we will always make the no-parallel-line assumption: there exist at least two angles $\phi,\psi \in A$, $\phi \neq -\psi$, such that with positive probability, there are motorcycles that travel in this direction. This assumption is meant to rule out the trivial case where all motorcycles travel in parallel to each other, and then clearly their paths are of infinite lengths.

\begin{prop}\label{lem:racs}
Under the no-parallel-line assumption, for all $k \geq 1$, $\mathcal{G}^k(\mathcal{P}, \mathcal{A})$ is a well-defined random closed set.

\end{prop}
\begin{proof}
For a motorcycle $\bar{a} = (a,\phi)$, let $L^k(\bar{a})$ denote the length of its path from its origin to its grave in $\G^k(\mathcal{P}, \mathcal{A})$. Let $\bar{\mathcal{P}}$ denote the marked point process consisting of ground points $\mathcal{P}$ and independent marks defined by $\A$. The goal is to show that $L^k$ is an exponentially stabilizing
functional of $\bar{\mathcal{P}}$. That is, for each $\bar{a}$,
there exists an a.s. finite random variable 
$R(\bar{a}, \bar{\mathcal P}) \in \R_{\geq 0}$ such that 
$L^k(\bar{a})$ is a finite random variable
only depending on the points
of $\bar{\mathcal P}$ which are inside a ball centered at $a$ and with radius $R(\bar{a}, \bar{\mathcal P})$.
In addition, the tail of the random variable
$R(\bar{a},\bar{\mathcal P})$ is exponential.
This in turn implies that the union of the edges
in $\mathcal{G}^k$ is a well-defined random closed set in $\R^2$
in view of the fact that
the support of $\mathcal{P}$ has no finite accumulation points.

Schreiber and Soja \cite{schreiber2010limit} proved this statement for the classical Gilbert 
process \cite[Theorem 4]{schreiber2010limit}.
Their proof only requires a small modification to adapt to this general setting.
Indeed, consider a marked point $\bar{a}$. The crux of the proof for $k=1$ is to show that there exists a
positive probability $\epsilon > 0$ such that the motorcycle 
has its grave in the unit ball $B(a,1)$,
or equivalently $\P(L^1(\bar{a}) > 1) \leq (1-\epsilon)$.
Inside $B(a,1)$, choose a region $D$ such that, regardless of the configuration of points outside $B(a,1)$, 
\begin{itemize}
\item for each $b \in D$, there exists a positive probability
$\epsilon_1 > 0$ such that the motorcycle $\bar{b} = (b,\psi)$ whose random direction $\psi$ is sampled according to $\A$ kills $a$ in the event $E$ that disregarding multiplicity, $a,b$ are the only two points of $\mathcal{P}$ in $B(a,1)$, and
\item $D$ has positive area, so with probability $\epsilon_2 > 0$, 
event $E$ holds.
\end{itemize}
Since the set of angles $A$ is finite, the angles at different points are chosen independently, and by the assumption which rules out parallel lines, the desired region $D$ exists. 
Then, with probability $\epsilon = \epsilon_1\epsilon_2 > 0$, the motorcycle $\bar{a}$ is killed by some motorcycle $\bar{b}$ with $b \in D$ and appropriate travel angle. Consider now non-overlapping and contiguous balls of radius 1
along the half line of apex $a$ and direction $\phi$.
Define $N$ to be the first integer such that the ball $B(x+2N,1)$
contains a single marked point $\bar{b}$ in $D+2N$ with an appropriate angle such that $\bar{b}$ would kill $\bar{a}$ if latter is still alive
by the time they are supposed to meet. Set our stabilizing radius 
$R(\bar{x},\bar{\mathcal{P}}) = 2N + 1$. By the independence property of $\mathcal{P}$, one has $\P(R > n) \leq (1-\epsilon)^n$, and this supplies the finiteness and the exponential decaying behavior needed.

Now consider the case $k > 1$. Let $N$ now denote the smallest integer
such that there are $k$ ball with the appropriate single point property in the sequence of balls $B(x+2N,1)$.
Clearly $2N+1$ is a stabilizing radius with an exponential tail.
\end{proof}

We have sacrificed generality for readability in Definition \ref{defn:gk}. Indeed, one can do away with several of our initial assumptions on $\mathcal{P}, \A$ and $\mathcal{M}$, and Proposition \ref{lem:racs} still holds by the same proof. In Section \ref{subsec:expanded}, we explore a non-trivial extension where at time $t = 0$, there are obstacles in $\R^2$.  

\subsection{Iterated Gilbert Mosaic}

In this paper, we shall focus on models where $\mathcal{G}^k(\mathcal{P}, \mathcal{A})$ is a random mosaic. This is a countable system of compact, convex polygons that covers $\R^2$, with mutually no common interior points.
Such a random mosaic can be identified with a tuple of point processes consisting of the centroids of its facets, edges and vertices. Applications of Campbell's formula and Euler's formula allow one to do computations on the statistics of random mosaics, see \cite[\S 10]{schneider2008stochastic}. 

\begin{prop}\label{prop:gk.mosaic}
Consider an iterated Gilbert model $\G^k = \mathcal{G}^k({\mathcal P}, \A)$ that satisfy the no-parallel-line assumption, and furthermore,
\begin{itemize}
  \item (no isolated sites): $\A^1$ has total measure 0;
  \item (convex sites): for a supported value $2 \leq m \leq M$, the joint angle distribution $\A^m$ is such that the absolute value of an angle formed between adjacent lines is less than or equal to $\pi$.  
\end{itemize}
Then for $k \geq 1$, $\G^k$ is a random mosaic of the plane. 
\end{prop}
\begin{proof}
By Proposition \ref{lem:racs}, $\G^k$ is a well-defined random closed set of $\R^2$. We say that two points of $\R^2\setminus \G^k$ are connected if there is a finite continuous path between them that does not intersect $\G^k$.
This is an equivalence relation on $\R^2\setminus \G^k$. The
equivalence classes are open sets which we will call cells.
We need to show that cells are a.s. relatively compact and convex,
and that their closures cover $\R^2$. As a random closed set, $\G^k$ consists of a.s. finite line segments. As there are no accumulation points in $\mathcal{P}$, only finitely many segments intersect any given compact set. So the closures of the cells of $\G^k$ cover $\R^2$, and furthermore, locally at each vertex, the cell is a polygon. 
 We now prove convexity. Suppose for contradiction that there exists a cell of $\mathcal{G}^k$ that is not convex. As it is locally a polygon, it has a vertex with interior angle greater than $\pi$. We claim that a.s. no such vertices exist in $\mathcal{G}^k$. Indeed, a point of $\G^k$ is a vertex of some cell if and only if it is a point of $\mathcal{P}$, or it is a location where some motorcycle hits the line of another. In the former case, by the (no isolated sites) and (convex sites) assumptions, the interior angle is at most $\pi$. In the later case, since $\mathcal{P}$ is in general position, at least one motorcycle must continue after the collision, thus the point lies in the relative interior of at least one of the two lines.
So the interior angle is also at most $\pi$. This proves the claim. Therefore, all facets of $\mathcal{G}^k$ are a.s. convex. 
Finally, we prove compactness. Let $F$ be a facet of $\G^k$. 
Let $F'$ be the polygon obtained by removing all vertices of $F$ with flat interior angles. The edges of $F'$ must be parallel to one of the angles in $A$. Since $F'$ is a convex polygon, it can contain at most two edges with the same angle. But there are $M$ different angles, thus $F'$ has at most $2M$ edges.
Since the points of $\mathcal P$ are in general position, each edge of $F'$ is a.s. generated by one motorcycle. Thus, the length of an edge in $F'$ is at most the distance that this motorcycle travels in $\G^k$ before dying. The later is a.s. finite by Proposition \ref{lem:racs}. So $F'$ is a.s. compact. Thus, $F$ is a.s. compact. 
\end{proof}

A mosaic is said to be \term{face-to-face} if the facets form a cell complex, that is, the boundaries of facets have mutually no common interior points.
Our iterated Gilbert model above is not face-to-face: an edge may terminate at an interior point of another edge. This issue is simple to resolve: one simply counts such interior points as vertices of the new edge, and define an edge as the line segment between two vertices, as before. This allows vertices with flat (180 degree) angles, and consecutive edges which are parallel to each other. This operation is called a face-to-face refinement. 
We can now define the central object of our study, the iterated Gilbert mosaic.
\begin{defn}
An \emph{iterated Gilbert mosaic} is the face-to-face refinement of an iterated Gilbert model $\G^k = \mathcal{G}^k({\mathcal P}, \A)$ that satisfies the assumptions of Proposition \ref{prop:gk.mosaic}.
\end{defn}

\begin{defn}\label{defn:vertex.type}
For $k \geq 1$, say that a vertex $v$ of $\mathcal{G}^k$ is a \emph{site} if $v \in \mathcal{P}$, and an \emph{intersection} if $v$ is the intersection of a line with another line. 
\end{defn}

\begin{prop}
\label{prop:gk-finite}
Let $\G^k = \mathcal{G}^k({\mathcal P}, \A)$ be an iterated Gilbert mosaic. Let $\E \mathcal{M}$ denote the mean multiplicity at a point in $\P$. For all $k \geq 1$, let $\lambda^k_0,\lambda^k_1,\lambda^k_2$ denote the (possibly infinite) intensities of the vertex, edge and facet processes of $\G^k$. Then
\begin{align*}
\lambda^k_0 &= (1 + k)(\E\mathcal{M})\lambda, \\
\lambda^k_1 &= \frac{(\E\mathcal{M})+4k-1}{2}(\E\mathcal{M})\lambda, \\
\lambda^k_2 &= \frac{(\E\mathcal{M})+2k-3}{2}(\E\mathcal{M})\lambda.
\end{align*}
In particular, $\G^k$ has all its facet sub-processes with finite intensity.
\end{prop}
\begin{proof}
First consider $k = 1$. Vertices of $\G^1$ are either sites or intersections. The intensity of sites is $\lambda (\E\mathcal{M})$. Each intersection corresponds to precisely one death event of a motorcycle. Since each motorcycle dies exactly once, the intensity of intersections is also $\lambda (\E\mathcal{M})$. Thus $\lambda^1_0 = 2(\E\mathcal{M})\lambda$.
Now consider the edge process of $\G^1$. For this, we use a mass transport argument. Construct a directed graph $G$ as follows: the vertices of this graph are the vertices of $\G^1$ and the centroids of the edges of $\G^1$. From each edge centroid, put a directed edge to each of the two vertices of this edge. Note that $G$ is a bipartite graph, from the set of edges of $\G^1$ to the set of vertices of $\G^1$. As each edge of $\G^1$ generates precisely two directed edges, the mean out-degree $\delta_{out}$ of $G$ is
$$ \delta_{out} = 2\lambda_1^1. $$
Now consider the mean in-degree of $G$. Each site in $\G^1$ contributes a mean in-degree of $(\E\mathcal{M})$. Each intersection contributes an in-degree of 3, by general positioning of the points in ${\mathcal P}$. Therefore, the mean in-degree $\delta_{in}$ of the graph $G$ is
$$ \delta_{in} = (\E\mathcal{M})(\E\mathcal{M})\lambda + 3(\E\mathcal{M})\lambda = (3+(\E\mathcal{M}))(\E\mathcal{M})\lambda. $$
The mass transport principle says that $\delta_{in} = \delta_{out}$. Hence $\lambda_1^1 = \frac{3+(\E\mathcal{M})}{2}(\E\mathcal{M})\lambda$ as claimed.
We use the same argument to derive the formula for $\lambda_2^1$. Construct a directed graph $G'$ as follows: the vertices of this graph are the vertices of $\G^1$ and the centroids of the facets of $\G^1$. From each facet centroid, put a directed edge to each of the vertices of this facet. Note that $G'$ is a bipartite graph, from the set of edges of $\G^1$ to the set of vertices of $\G^1$.
The mean in-degree of $G'$ is
$$ \delta_{in} = \delta \lambda_2^1 $$
for some constant $\delta > 0$, interpreted as the mean number of vertices per face of $\G^1$. Now consider the mean out-degree $\delta_{out}$ of $G'$. For a vertex of $\G^1$, the number of faces with this vertex equals the number of edges at this vertex. So the mean in-degree of $G'$ equals to the mean in-degree of $G$, which is
$$ \delta_{in} = (3+(\E\mathcal{M}))(\E\mathcal{M})\lambda. $$
By the mass transport principle,
$$ (3+(\E\mathcal{M}))(\E\mathcal{M})\lambda = \delta\lambda_2^1. $$
Since $(\E\mathcal{M}),\lambda < \infty$, the quantities on the right-hand side must also be finite. This implies that $\G^k$ is a random mosaic with finite intensity. By the Euler characteristic formula \cite[Equation 14.63]{schneider2008stochastic}, 
$$ \lambda_2^1 = \lambda_1^1 - \lambda_0^1.$$
Rearranging gives the formula for $\lambda_2^1$. 
The case of general $k$ is similar. Here a motorcycle dies precisely $k$ times, hence $\lambda_0^k = (k+1)(\E\mathcal{M})\lambda$. Note that each motorcycle has one final death event, which corresponds precisely to one intersection of degree 3. For all other collisions, the two motorcycles involved will continue, creating vertices of intersection of degree 4. Thus, each motorcycle creates $k-1$ vertices with multiplicity 4, and $1$ vertex with multiplicity $3$. This implies the equation
$$ 2\lambda_1^k = (\E\mathcal{M})(\E\mathcal{M})\lambda + 4(k-1)(\E\mathcal{M})\lambda + 3(\E\mathcal{M})\lambda = (4k-1+(\E\mathcal{M}))(\E\mathcal{M})\lambda. $$
Finally, for the facets, by the Euler characteristic formula,
$$ \lambda_2^k =  \lambda_1^k - \lambda_0^k = \frac{(\E\mathcal{M})+2k-3}{2}(\E\mathcal{M})\lambda.$$
\end{proof}

\begin{cor}
Let $\G^k = \mathcal{G}^k({\mathcal P}, \A)$ be an iterated Gilbert mosaic. Let $\E \mathcal{M}$ denote the mean multiplicity at a point in $\mathcal P$. Then the mean number of vertices per face of $\G^k$ is $2\lambda(\E\mathcal{M})\frac{(\E\mathcal{M})+4k-1}{(\E\mathcal{M})+2k-3}$.
\end{cor}

\section{Scaling limits of iterated Gilbert mosaics}\label{sec:main.proof}

Let $\G^k$ be an iterated Gilbert mosaic.
We want to know if there exists a sequence $f(k)$ such that
when taking for the intensity of $\mathcal P$
$\frac{\lambda} {f(k)}$ rather than $\lambda$, $\G^k$
converges (in some sense) to a non-trivial limiting random mosaic. Proposition \ref{prop:gk-finite} suggests that one should take $f(k) = k$ to see non-trivial limits. Our main result, Theorem \ref{thm:main}, states that at this scaling, the limit in the vague topology is a Poisson line process with a particular measure. To state this limiting measure, we first need some definitions. 

For $w = \{w_\phi \in \R_{\geq 0}: \phi \in A\}$, view $w$ as a vector in $\R^M$ with non-decreasing coordinates, that is, $w_1 \leq \ldots \leq w_M$. For a pair $\phi,\varphi \in A$, let $T^{w_\phi,w_\varphi}_{\phi\varphi} \subset \R^2$ be the polygon with vertex set
\begin{align*}
\mathsf{vertex}(T^{w_\phi,w_\varphi}_{\phi\varphi}) =
\left\{
\begin{array}{cc}
\{(0,0), -w_\varphi\cdot\vect{\varphi},  -w_\phi\cdot\vect{\phi}, (w_\varphi-w_\phi)\cdot \vect{\phi}-w_\varphi\cdot\vect{\varphi} \} &\mbox{ if } 
w_\varphi < w_\phi, \\
\{(0,0), -w_\phi\cdot\vect{\varphi},  -w_\phi\cdot\vect{\phi}\} &\mbox{ else.} 
\end{array}
\right.
\end{align*}
That is, $T^{w_\phi,w_\varphi}_{\phi\varphi}$ is a triangle with side lengths $w_\phi$ if $w_\varphi \geq w_\phi$, and otherwise, it is the trapezium obtained by truncating a piece off the triangle with side lengths $w_\varphi$. Note that $T^{w_\phi,w_\varphi}_{\phi\varphi} \neq T^{w_\varphi,w_\phi}_{\varphi\phi}$ unless if $w_\phi = w_\varphi$. See Figure \ref{fig:regionT} for an illustration. 
\newpage
\begin{figure}[h!]
\includegraphics[width=0.65\textwidth]{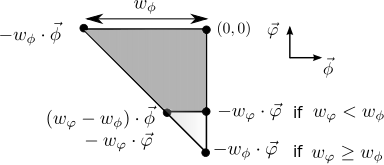}
\caption{The polygon $T^{w_\phi,w_\varphi}_{\phi\varphi}$ illustrated for the two directions $\vec{\varphi} = (0,1)$ and $\vec{\phi} = (1,0)$. If $w_\varphi < w_\phi$, then $T^{w_\phi,w_\varphi}_{\phi\varphi}$ is the shaded trapezium. If $w_\varphi \geq w_\phi$, then $T^{w_\phi,w_\varphi}_{\phi\varphi}$ is the entire triangle, that is, the trapezium union with the white small triangle at the bottom. For $\bar{b} = (b,\phi)$, $\tau_b \circ T^{w_\phi,w_\varphi}_{\phi\varphi}$ is the polygon translated so that the vertex $-w_\phi\cdot\vec{\phi}$ is at $b$.}
\label{fig:regionT}
\end{figure}

For all $\phi\in A$, define
\begin{equation}
\label{eqn:defn.e}
\mathcal{E}(w,\phi) = 
\sum_{Q \subseteq A \setminus \{\phi\}}
\sum_{\varphi \in Q}
|T^{w_\phi,w_\varphi}_{\phi \varphi}|(\mu_{\phi \cup Q}+\mu_{Q}),
\end{equation}
where we recall that $\mu_{Q}$ denotes the intensity of sites in $\mathcal{P}$ with set of angles $Q$. Assume that each motorcycle $\bar{a} = (a,\varphi)$ travels a distance of exactly $w_\varphi$; then $\mathcal{E}(w,\phi)$ is the mean number of motorcycle trajectories that motorcycle $\bar{b} = (b,\phi)$ crosses on $[b, b + w_\phi \vec \phi]$ such that $\bar{b}$ does not come first at this
intersection, see Figure~\ref{fig:angles}. The two terms $\mu_{\phi \cup Q}$ and $\mu_{Q}$ take into account the fact that motorcycle $\bar{b}$ crosses trajectories stemming from sites with angle set containing $\phi$ as well as from sites not containing $\phi$.

In addition to (\ref{eqn:defn.e}), it is sometimes convenient to consider another formula for $\mathcal{E}(w,\phi)$. For each $S\subset A\setminus \{\phi\}$, define
\begin{equation}
T^w_{\phi S} := \left( \bigcap_{\varphi \in S}
T^{w_\phi,w_\varphi}_{\phi\varphi}\right) \backslash
\left(\bigcup_{{Q' \supset S}\atop{ Q'\ne S}}
\bigcap_{\varphi' \in Q'}T^{w_\phi,w_{\varphi'}}_{\phi\varphi'}\right) \subset \R^2 \label{eqn:tw}.
\end{equation}
Then one can rewrite $\mathcal{E}(w,\phi)$ as
\begin{equation}
\mathcal{E}(w,\phi) = \sum_{Q \subseteq A \backslash \{\phi\}}
(\mu_{\phi \cup Q} + \mu_Q)
\sum_{S\subseteq Q}
|S||T^w_{\phi S}|
.\label{eqn:defn.e2}
\end{equation}
The advantage of the last formulation is that by definition,
the sets $T^w_{\phi S}$ are mutually disjoint.
So $\mathcal{E}(w,\phi)$ is the expected value
of the weighted sum of independent Poisson random variables
$$ \sum_{Q \subseteq A \backslash \{\phi\}} 
\sum_{S\subseteq Q}
|S|(\mathcal{P}_{\phi \cup Q} \left(T^w_{\phi S}) +
\mathcal{P}_{Q} (T^w_{\phi S})\right). $$

\begin{figure}[h]
\centering{
\resizebox{0.90\textwidth}{!}{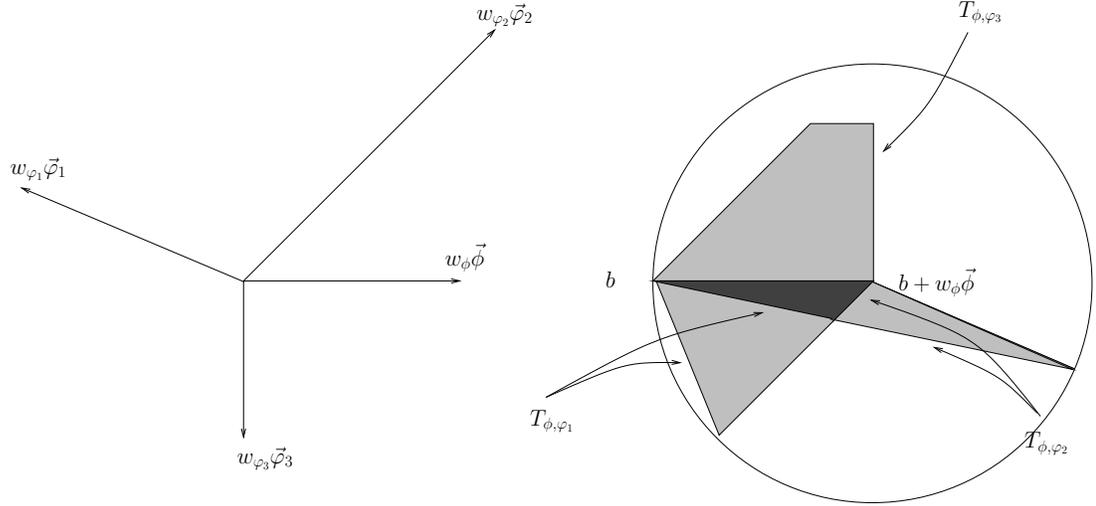}}
  \caption{Illustration of Formula (\ref{eqn:defn.e}). Left: the set of angles depicted with the distances $w$. Right: the three associated regions. Here, $\mathcal{E}(w,\phi)$ is the average number of sites of $\mathcal{P}$ which contains a motorcycle $\bar{a} = (a,\varphi)$, $\varphi \neq \phi$, that comes first at their meeting with $\bar{b} = (b,\phi)$ on $[b, b+w_\phi\vec{\phi}]$. Note that points $T_{\phi\varphi_1} \cap T_{\phi\varphi_2}$ are counted twice in the sum. Regroup the sum by the mutually disjoint intersected regions, weighted by their multiplicities, gives (\ref{eqn:defn.e2}).}
  \label{fig:angles}
\end{figure}

 Consider now the iterated Gilbert model and 
the distance a typical motorcycle $\bar{a} = (a,\varphi)$ can
travel with $k$ lives. Assuming that this distance 
is concentrated around some mean value $w^\ast_\varphi\sqrt{k}$ for
all $\varphi$, 
the vector $w^\ast \in \R^M$ should satisfy the relation
\begin{equation}\label{eqn:constants}
\mathcal{E}(w^\ast_\phi\sqrt{k},\phi) = k, \mbox{ for all }\phi \in A.
\end{equation}

\begin{lem}\label{lem:constants}
There exists a unique set of positive constants $w^\ast = \{w^\ast_\phi \in \R_{\geq 0}: \phi \in A\}$ that satisfies (\ref{eqn:constants}) for all $k > 0$. 
\end{lem}
\begin{proof}
Since $|T^{k \cdot w}_{\phi Q}| = k^2 |T^w_{\phi Q}|$ for any $k > 0$, for any set of constants $w = \{w_\phi\}$,
$$ \frac{\mathcal{E}(w_\phi\sqrt{k}, \phi)}{k} = \mathcal{E}(w_\phi, \phi) \mbox{ for each } \phi \in A. $$
So one just needs to show that there exists a unique set of constants $w^\ast$ such that
$$ \mathcal{E}(w^\ast_\phi,\phi) = 1 \mbox{ for all } \phi \in A.$$
Consider the map $\R_{\geq 0} \to \R_{\geq 0}^M$, $c \mapsto (\mathcal{E}(c \cdot \mathbf{1},\phi), \phi \in A)$. This map is continuous, increasing and tends to infinity in each of its coordinates. Thus, there exists a unique constant $w^\ast_1 \in \R_{\geq 0}$ such that
\begin{equation}\label{eqn:wstar1}
\max_{\phi \in A} \mathcal{E}(w^\ast_1 \cdot \mathbf{1},\phi) = 1. 
\end{equation}
Let $\phi_1 \in A$ be an angle that achieves this maximum, that is, 
$$\mathcal{E}(w^\ast_1 \cdot \mathbf{1},\phi_1) = 1.$$
It follows from the definition of $\mathcal E$ (and more precisely the trapezium structure when  $\varphi>\phi$ in the definition of $T^{w_\phi,w_\varphi}_{\phi \varphi}$), that if $w = \{w_\phi\}$ is such that $w_{\phi_1} = w^\ast_1$, and $w_\varphi \geq w^\ast_1$ for all $\varphi \in A$, then
$$ \mathcal{E}(w,\phi_1) = 1. $$
So now, let $\mathbf{1}^{(-1)} \in \R^M$ be the all-one vector, except in the coordinate corresponds to $\phi_1$, where it is $0$. For $c\ge 0$, the map $\R_{\geq 0} \to \R_{\geq 0}^M$, $c \mapsto (\mathcal{E}(w^\ast_1 \cdot \mathbf{1} + c \cdot \mathbf{1}^{(-1)}\}, \phi), \phi \in A)$
is constant in the coordinate corresponds to $\phi_1$, while in other coordinates, it is continuous, monotone increasing, with starting value at most 1. Thus, there exists a unique constant $w_2^\ast \geq w_1^\ast$ such that 
$$ \max_{\phi \in A\backslash\{\phi_1\}} \mathcal{E}(w^\ast_1 \cdot \mathbf{1} + (w^\ast_2-w^\ast_1) \cdot \mathbf{1}^{(-1)}\}, \phi) = 1,$$
and in particular, this maximum is achieved at some coordinate $\phi_2 \in A \backslash \{\phi_1\}$. Repeat this argument, we obtain the unique $w^\ast$ needed.
\end{proof}

Let $\G^k\left(\frac{1}{k}\right)$ denote the the iterated Gilbert mosaic of order $k$ for the
same angle distribution $\A$ as above, but for
a Poisson point process 
with intensity multiplied by $\frac{1}{k}$.
For all $R\subset \R^2$ and $u\in \R$, let
$u\cdot R$ denote the set $\{ ux, x\in R\}$.
Note that since $\mathcal{P}$ is stationary, multiplying the intensity
by $\frac{1}{k}$ is the same as rescaling space by $\sqrt{k}$
in the $x$ and $y$ axes. Hence
$\G^k\left(\frac{1}{k}\right)$ is equal in distribution to
$\sqrt{k} \cdot \G^k$, which will be used throughout in what follows.

\begin{thm}\label{thm:main}
Let $w^\ast$ be the unique set of constants in Lemma \ref{lem:constants}. Let $\G^k({\mathcal P},\A)$ be an iterated Gilbert mosaic. As $k \to \infty$, for any compact window $W \subset \R^2$,
$$
\G^k\left( \frac{1}{k}\right)
\cap W \to \G^\infty\cap W \mbox{ in probability}, $$
where $\G^\infty$ is a Poisson line process with cylindrical measure
$\Lambda dr \times \Theta(d\theta)$, with $\Lambda$ the constant
$$ \Lambda =  \sum_{\phi\in A}
w^*_\phi \sum_{Q\subset A, \phi\in Q} \mu_Q
$$
and $\Theta$ the probability measure with mass 
$$\frac 1  \Lambda w^*_\phi \sum_{Q\subset A, \phi\in Q} \mu_Q$$
at $\phi^\perp=\phi+\frac \pi 2$ (recall that lines are parameterized by their point 
which is the closest to the origin, and that the angle of this point is $\phi^\perp$ 
if the line has angle $\phi$)
for all $\phi \in A$, where $w^\ast$ is defined by (\ref{eqn:constants}).
\end{thm}

Let us now explain qualitatively why the mosaics admit a scaling
limit at a linear rate. As we saw in Proposition \ref{prop:gk-finite},
the intersections of $\G^k$ densify at a linear rate with respect to $k$.
Suppose we knew that the limit 
$\G^k \left(\frac{1}{k}\right)$ exists.
Intuitively, the limiting process must be a classical Poisson line process. 
Then, the starting points of the motorcycles, which are points of
$\mathcal{P}$, are getting further and further apart. 
A view of the process by a typical compact window $W$
consists of paths of the motorcycles, which are lines with directions in $A$.
With high probability, these lines are independent, since they come
from different, far-away starting points. Thus, the limiting process
must be a classical Poisson line process.
The difficulties are in working out the measure $\Theta$ 
precisely and in proving that the limit holds indeed.
\vskip12pt
We now state and prove the two auxiliary results, Propositions \ref{prop:concentrate} and \ref{prop:cutoff}, used for the proof of Theorem \ref{thm:main}. Fix~$k \in \mathbb{N}$. 
For a motorcycle $\bar{b} = (b,\phi)$, recall that $L^k(\bar{b})$ is the length of the path from its origin to its grave in $\G^k$. Proposition \ref{prop:concentrate} claims that for large $k$, $L^k(\bar{b})/\sqrt{k}$ concentrates around $w^\ast_\phi$, and this concentration holds simultaneously for all motorcycles whose starting points lie in some dilated compact set. 

\begin{prop}\label{prop:concentrate}
Fix a compact set $R \subset \R^2$. With probability $1-\epsilon(k,\phi)$ which approaches $1$ as $k \to \infty$,
$$ \sup_{b \in \sqrt{k} \cdot R}\left|\frac{L^k((b,\phi))}{\sqrt{k}} - w^\ast_\phi\right| \leq o(1). $$
\end{prop}

We defer the proof of Proposition \ref{prop:concentrate} to the next section. The heart of the argument uses a Chernoff-type bound to control the supremum of a Poisson functional, and an induction on the sequence of angles of $A$, ordered such that the sequence of constants $w^\ast = (w^\ast_\phi)$ is non-decreasing. 

Now fix $\phi \in A$. Let $W \subset \R^2$ be a compact set. For $k \geq 1$, define $W_k := \frac{1}{\sqrt{k}} \cdot W$. The second auxiliary result is Proposition \ref{prop:cutoff} below, which claims that for large $k$, with overwhelming probability, motorcycles whose paths intersect $W_k$ in $\G^k$ must have their origins in a particular strip. 

\begin{prop}\label{prop:cutoff}
With probability $1 - \epsilon'(k,\phi)$ which approaches $1$
as $k \to \infty$,
the path in $\G^k$ of a motorcycle
$\bar{b} = (b,\phi)$ intersects $W_k$ if and only if 
$$ b \in R(\phi,k) := [\sqrt{k}w^\ast_\phi\cdot(-\vect{\phi}),(0,0)]
\oplus W_k, $$
where for all $C,D\subset \R^M$, $C\oplus D$ denotes the set
$\{x+y,x\in C,y\in D\}$.
\end{prop}

\begin{proof}
Assume without loss of generality that $W$ is convex. Fix $k$.
Let $\Pi_\phi W_k$, $\Pi_{\phi^\perp} W_k$ be the line segments
obtained by projecting $W_k$ along $\phi$ and $\phi^\perp$,
onto $\R \cdot \vect{\phi^\perp}$ and $\R \cdot \vect{\phi^\perp}$,
respectively. Let $g_k$ and $d_k$ be the left-most and right-most points of $\Pi_{\phi^\perp} W_k$, that is, 
\begin{align*}
g_k &= \sup\{c \in \R: \langle w, \phi \rangle \geq c \mbox{ for all } w \in W_k\}, \\
d_k &= \inf\{c \in \R: \langle w, \phi \rangle \leq c \mbox{ for all } w \in W_k\}.
\end{align*}
For $y > 0$, let $R(y)$ be the translated segment
$y\sqrt{k}\cdot(-\vect{\phi}) + \Pi_\phi W_k - g_k$. 
Note that any $\bar{b}$ whose path intersects $W_k$
must have $b \in R(y)$ for some $y > g_k-d_k$.
To prove the statement, it is sufficient to show that
with high probability, the following events hold simultaneously
(see Figure \ref{fig:cutoff2}):
\begin{align}
&b \in R(y), 0 < y < w^\ast_\phi-o(1) \Rightarrow \mbox{ $\bar{b}$ hits $W_k$,} \label{eqn:hit}  \\
&b \in R(y), y > w^\ast_\phi+o(1) \Rightarrow \mbox{ $\bar{b}$ does not hit $W_k$,} \label{eqn:no.hit} \\
&\mbox{ there are no motorcycles $b$ with $b \in R(y)$ for }w^\ast_\phi-o(1) \leq y \leq w^\ast_\phi+o(1) \mbox{ or } g_k-d_k \leq y \leq 0. \label{eqn:no.point}
\end{align}
By Proposition \ref{prop:concentrate}, with probability at least $1-2\epsilon(k,\phi)$, 
\begin{align}
L^k(\bar{b}) &> \sqrt{k}(w^\ast_\phi-o(1)) \mbox{ for all } b \in R(w^\ast_\phi-o(1)), \mbox{ and } \label{eqn:hit.1} \\
L^k(\bar{b}) &< \sqrt{k}(w^\ast_\phi+o(1)) \mbox{ for all } b \in R(w^\ast_\phi+o(1)). \label{eqn:no.hit.1}
\end{align}
Now, consider shifting motorcycle $\bar{b} = (b,\phi)$ along $\vect{\phi}$ to a starting location $b'$ closer to $W_k$, while keeping all other motorcycles the same. 
Clearly if $\bar{b}$ can hit $W_k$ from $b$, it also can hit $W_k$ from $b'$. Thus, (\ref{eqn:hit.1}) implies (\ref{eqn:hit}). Similarly, if $b'$ is further away from $W_k$, then if $\bar{b}$ cannot hit $W_k$ from $b$, it also cannot hit $W_k$ from $b'$. Thus, (\ref{eqn:no.hit.1}) implies (\ref{eqn:no.hit}). Finally, as $W$ is compact, $|g_k-d_k| = O(k^{-1/2})$, so the last event (\ref{eqn:no.point}) is contained in the event that there is no point of $\mathcal{P}$ in a region with area $o(1)$, so it happens with probability $1-\epsilon_k$ for $\epsilon_k \to 1$ as $k \to \infty$. So, with probability at least $1 - 2\epsilon(\phi,k)-\epsilon_k$, the desired events (\ref{eqn:hit}-\ref{eqn:no.point}) hold. Choose $\epsilon'(k,\phi) = 2\epsilon(\phi,k)+\epsilon_k$, one obtains the desired result. 
\end{proof}

\begin{figure}[h]
\includegraphics[width=0.85\textwidth]{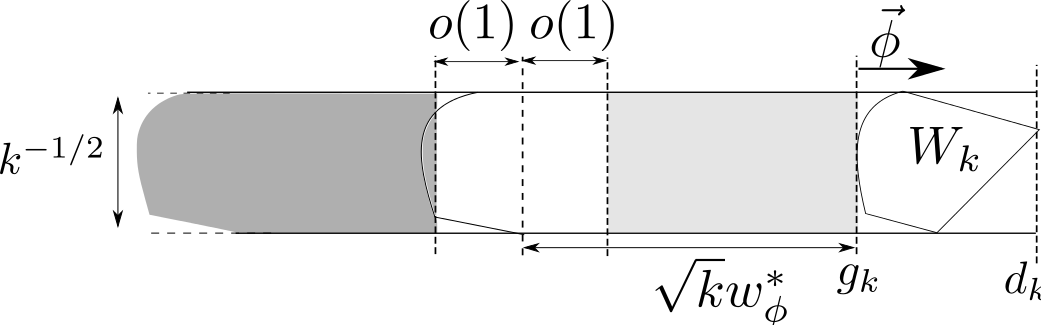}
\caption{The cut-off phenomenon in Proposition \ref{prop:cutoff} illustrated. With high probability, in $\G^k$, all points $\bar{b} = (b,\phi)$ with starting point $b$ in the light region will hit $W_k$, all those with $b$ in the dark region will not hit $W_k$, and there are no points in the white region.}
\label{fig:cutoff2}
\end{figure}

\begin{proof}[Proof of Theorem \ref{thm:main}]
Let $\{P_\phi, \phi \in A\}$ be independent Poisson line processes, 
with $P_\phi$ consisting of lines parallel to $\vect{\phi}$,
whose projection onto $\R \cdot\vect{\phi}^\perp$
form a Poisson point process with intensity $ w^*_\phi \sum_{Q\subset A, \phi\in Q} \mu_Q$. Note that $\G^\infty = \bigcup_{\phi \in A} P_\phi$.
By Proposition \ref{prop:cutoff}, as $k \to \infty$, the process of segments of $\G^k$ parallel to $\phi$ 
that intersect $\frac{1}{\sqrt{k}}W$
converges in probability to the
process of lines of $P_\phi$ that intersect $W$. As $A$ is a finite set, by union bound over $A$, with high probability, the events in Proposition \ref{prop:cutoff} hold simultaneously for all $\phi \in A$. Let
$$R'(\phi,k) = R(\phi,k) \backslash \bigcup_{\varphi \in A, \varphi \neq \phi} \left(R(\phi,k) \cap R(\varphi,k)\right).$$ 
Since $W$ is compact, the pairwise intersections $R(\phi,k) \cap R(\varphi,k)$ has area of order $O(k^{-1/2})$, while $R(\phi,k)$ has area of order $O(1)$. 
So with high probability, $\mathcal{P} \cap R(\phi,k) = \mathcal{P} \cap R'(\phi,k)$ for all $\phi \in A$. Since the regions $R'(\phi,k)$ are pairwise disjoint, the lines intersecting $\frac{1}{\sqrt{k}}W$ in $\G^k$ converges in probability to the intersection of $W$ and $\G^\infty$. That is,
$$\lim_{k \to \infty}\G^k \cap \frac{1}{\sqrt{k}}W \stackrel{P}{\to} \G^\infty \cap W. $$
As $\G^k$ is stationary, this implies
$$\lim_{k \to \infty}
\G^k\left( \frac 1 {k}\right) 
\cap W \stackrel{P}{\to} \G^\infty \cap W, $$
and this concludes the proof of Theorem \ref{thm:main}.
\end{proof}

\subsection{Proof of Proposition \ref{prop:concentrate}}

The proof is organized in a series of lemmas.
We start with a concentration result on Poisson point processes to be used
in the proofs.

\begin{lem}\label{lem:sup.bound}
Let $\mathcal{P}_k$ be a PPP with rate $k\lambda$. Let $R, S \subset \R^2$ be compact sets. Assume that $S$ has finite boundary, that is, $|S^t \backslash S| = O(t)$ as $t \to 0$, with $S^t = \{x \in \R^2: \|x - S\|_\infty \leq t\}$ where $\|y\|_\infty := \max\{|y_1|,|y_2|\}$ is the $L_\infty$-norm in $\R^2$. For $a \in R$, let $N_k(a) = \mathcal{P}_k (S + a)$ be the number of points of $\mathcal{P}_k$ in the set $S + a$. For any fixed $\epsilon > 0$, as $k \to \infty$, with high probability,
$$ \sup_{a \in R} \left|N_k(a) - k\lambda|S| \right| \leq O(\lambda|S|k^{1/2+\epsilon}). $$ 
In other words, with high probability,
$$ k\lambda|S| + O(\lambda|S|k^{1/2+\epsilon})\leq \inf_{a \in R}N_k(a)\leq \sup_{a \in R}N_k(a) \leq k\lambda|S| + O(\lambda|S|k^{1/2+\epsilon}). $$
\end{lem}
\begin{proof}
The proof is a union bound over a $\delta$-net.
Let $X$ be a Poisson random variable with mean $k\mu$.
By Chernoff's bound for Poisson random variables,
for $\epsilon > 0$, there exists a constant $C > 0$ such that
$$ \P(|X - k\mu| \geq (k\mu)^{1/2+\epsilon}) \leq \exp(-C(k\mu)^{2\epsilon}). $$
Set $X = N_k$ with $\mu = \lambda|S|$.
Cover $R$ by a grid where each square has side length at most
$\delta = k^{-2}$. Let $G$ be the set of center points of
squares which have non-empty intersection with $R$.
Associate to each point $a \in R$
the center $g(a) \in G$ of the square it belongs to
(squares can be taken closed on the left/bottom and
open on the right/top to avoid ties). Then, a.s.,
\begin{align}
 \sup_{a \in R} \left|N_k(a) - k\mu \right| &\leq \sup_{a \in R}
\left(\left|N_k(a) - N_k(g(a))\right| + \left| N_k(g(a)) - k\mu \right|\right)
\nonumber \\
&\leq \sup_{g \in G}\sup_{a: \|a-a_G\|_\infty < \delta/2}
\left|N_k(a) - N_k(g)\right|  + \left| N_k(g) - k\mu \right|. \label{eqn:sup.ag}
\end{align}
Since $\|a-g\|_\infty < \delta/2$, we have
$S^\delta + a_G \supseteq S + a$.
Let $N_k^\delta(a_G) = |\mathcal{P}_k \cap (S^\delta + a_G)|$. Then
$$ \sup_{a: \|a-a_G\|_\infty < \delta/2}\left|N_k(a) - N_k(a_G)\right| \leq  N_k^\delta(a_G) - N_k(a_G) =  |\mathcal{P}_k \cap (S^\delta \backslash S + a_G)|. $$
So $N_k^\delta(a_G) - N_k(a_G)$ is a Poisson random variable 
with mean at most $C'k\lambda \delta$ for some constant $C' > 0$,
thanks to the assumption on the boundary.
The cardinality of $G$ is at most $2|R|\delta^{-2}$. So by the union bound,
$$ \P(\sup_{a_G \in G}N_k^\delta(a_G) - N_k(a_G) > 0) \leq 2|R|\delta^{-2}(1-\exp(-C'k\lambda\delta)) = O(k/\delta) = O(k^{-1}),  $$ 
and
$$ \P(\sup_{a_G \in G} \left| N_k(a_G) - k\mu \right| \geq (k\mu)^{1/2+\epsilon}) \leq 2|R|\delta^{-2}\exp(-C(k\mu)^{2\epsilon}) = \exp(-O(k^{2\epsilon}) + 4\log k). $$ 
This together with the bound (\ref{eqn:sup.ag}) imply that w.h.p.
$$ \sup_{a \in R} \left|N_k(a) - k\mu \right| \leq \sup_{a_G \in G} (N_k^\delta(a_G) - N_k(a_G)) + \sup_{a_G \in G} \left| N_k(a_G) - k\mu \right| \leq (k\mu)^{1/2+\epsilon} = O(k^{1/2+\epsilon}) . $$
\end{proof}

Below is another auxiliary result, which are bounds on the function $w \mapsto \mathcal{E}(w,\phi)$ under small perturbations. They follow from the geometry of the regions $T^w_{\phi Q}$. 
\begin{lem}\label{lem:function.e}
Fix $i \in \{1, \ldots, M\}$. Fix a sequence of constants $0 < \delta_1 \leq \ldots \leq \delta_{i-1} < \delta_i \leq \delta_{i+1} \leq \ldots \leq \delta_M$.
Define weight vectors $w, w' \in \R^M$ with
$$ w_\ell = w^\ast_\ell + k^{\delta_\ell} \mbox{ for } \ell = 1, \ldots, i-1, \hspace{1em} w_\ell = w^\ast_i - k^{\delta_i} \mbox{ for } \ell \geq i, $$
and 
$$ w'_\ell = w^\ast_\ell, \mbox{ for } \ell = 1, \ldots, i-1,
\hspace{1em} w'_\ell = w^\ast_i, \mbox{ for } \ell \geq i. $$
Then for $j > i$,
$$ \mathcal{E}(w, \phi_j) = \mathcal{E}(w', \phi_j) + o(k^{\delta_i}) - Ck^{\delta_i} $$
for some constant $C > 0$. Similarly, if 
$$ w_\ell = w^\ast_\ell - k^{\delta_\ell} \mbox{ for } \ell = 1, \ldots, i-1, \hspace{1em} w_\ell = w^\ast_i + k^{\delta_i} \mbox{ for } \ell \geq i, $$
then for $j > i$,
$$ \mathcal{E}(w, \phi_j) = \mathcal{E}(w', \phi_j) - o(k^{\delta_i}) + Ck^{\delta_i} $$
for some constant $C > 0$. 
\end{lem}
\begin{proof}
Consider the difference $|T^w_{\phi_j Q}| - |T^{w'}_{\phi_j Q}|$ for each subset $Q \subseteq A \backslash \{\phi_j\}$. For coordinates $\ell < i$, $w_\ell - w'_\ell = k^{\delta_\ell} = o(k^{\delta_i})$, and thus this contributes a positive term of order $o(k^{\delta_i})$ to this difference. Now consider coordinates $\ell \geq i$. Then $w_\ell - w'_\ell = -k^{\delta_i}$, and thus this contributes a negative term of order $O(k^{\delta_i})$. Sum over all such subsets $Q$ and use (\ref{eqn:defn.e2}) to obtain
$$\mathcal{E}(w, \phi_j) - \mathcal{E}(w', \phi_j) =  o(k^{\delta_i}) - Ck^{\delta_i}$$
for some constant $C > 0$. 
The second part follows similarly.
\end{proof}

Fix $k \in \mathbb{N}$, a motorcycle $\bar{b} = (b,\phi)$,
with $b \in \sqrt{k} \cdot R$, some distance $y > 0$, and angle 
$\varphi \in A, \varphi \neq \phi$.
Write $K(\bar{b},y,\varphi,k)$ for the number of
{\em would-be killers} of $\bar{b}$  on $b + [0,y\sqrt{k}\cdot\vect{\phi}]$
which travel in direction $\varphi \in A$.
That is, these are motorcycles $\bar{a} = (a,\varphi)$
whose path in $\G^k$ would cross $b + [0,y\sqrt{k}\cdot\vect{\phi}]$,
and $\bar{a}$ would have killed $\bar{b}$ if $\bar{b}$ had enough lives
to meet it, which will happen for large enough $k$. Our goal is to give a tight bound of the kind
\begin{equation}\label{eqn:sandwich}
\underline{K}(\bar{b},y,\varphi,k) \leq K(\bar{b},y,\varphi,k)\leq \overline{K}(\bar{b},y,\varphi,k) \hspace{1em} \mbox{for all $b \in \sqrt{k} \cdot R$, w.h.p.,}
\end{equation}
for some appropriately defined $\underline{K}$ and $\overline{K}$. 
Summing over $\varphi \in A \backslash \{\phi\}$ and taking a union bound, 
one obtains upper and lower bounds for the number of would-be killers
$K(\bar{b},y,k)$ of $b$  on $b + [0,y\sqrt{k}\cdot\vect{\phi}]$,
$$ \underline{K}(\bar{b},y,k) \leq K(\bar{b},y,k)\leq \overline{K}(\bar{b},y,k) \hspace{1em}
\mbox{for all $b \in \sqrt{k} \cdot R$, w.h.p.}, $$
where 
\begin{align}
\underline{K}(\bar{b},y,k) & := \sum_{\varphi \in A \backslash \{\phi\}} \underline{K}(\bar{b},y,\varphi,k) \\ 
K(\bar{b},y,k) & := \sum_{\varphi \in A \backslash \{\phi\}} K(\bar{b},y,\varphi,k), \\
\overline{K}(\bar{b},y,k) & := \sum_{\varphi \in A \backslash \{\phi\}} \overline{K}(\bar{b},y,\varphi,k). 
\end{align}

The definitions of the quantities $\underline{K}(\bar{b},y,\varphi,k)$ and $\overline{K}(\bar{b},y,\phi,k)$ are given below, depending on whether $w^\ast_\phi$ or $w^\ast_\varphi$ is greater. Let us motivate these definitions. For $b \in \R^2$, let $\tau_b \circ T^{w_\phi,w_\varphi}_{\phi\varphi}$ be the polygon $T^{w_\phi,w_\varphi}_{\phi\varphi}$ translated so that the vertex $-w_\phi \cdot \vect{\phi}$ is now at $b$. 
Say that $\bar{a} = (a,\varphi)$ is a {\em potential killer}
of $\bar{b}$ in direction $\varphi$ on $b + [0,y\sqrt{k}\cdot\vect{\phi}]$
if, provided that both have enough lives to progress of $y\sqrt{k}$,
the path of $\bar{a}$ will intersect that of $\bar{b}$ on the line
segment $b + [0,y\sqrt{k}\cdot\vect{\phi}]$, and at this intersection,
$\bar{a}$ will kill $\bar{b}$.
Potential killers should
not to be confused with would-be killers - the latter pertain to
properties of ${\mathcal G}^k$, whereas the former pertain to some
geometric properties associated with $y>0$:
$\bar{a} = (a,\varphi)$ is a potential killer of
$\bar{b} = (b,\phi)$ if and only if
$a \in \tau_b \circ T^{y\sqrt{k},y\sqrt{k}}_{\phi\varphi}$.  

So the number of potential killers of $\bar{b}$
on $b + [0,y\sqrt{k}\cdot\vect{\phi}]$ in the direction $\varphi$ is
\begin{equation}\label{eqn:tilde.k}
\tilde{K}(\bar{b},y,\varphi,k) = \sum_{Q \subseteq A \backslash \{\varphi\}}
\mathcal{P}_{\varphi \cup Q} 
(\tau_b \circ T^{y\sqrt{k},y\sqrt{k}}_{\phi \varphi}),
\end{equation}
whereas the total number of potential killers is
\begin{eqnarray}
\label{eqn:tilde.ki}
\tilde{K}(\bar{b},y,k) & = & \sum_{\varphi\in A \setminus \phi}
\sum_{Q \subseteq A \backslash \{\varphi\}}
\mathcal{P}_{\varphi \cup Q} 
(\tau_b \circ T^{y\sqrt{k},y\sqrt{k}}_{\phi \varphi})
\label{eqn:tilde.ki-fl}
\\
& = &
\sum_{Q \subseteq A \backslash \{\phi\}}
\sum_{\varphi \in Q} 
(\mathcal{P}_Q
+ \mathcal{P}_{Q\cup\phi}) (\tau_b \circ T^{\sqrt{k}y,\sqrt{k}y}_{\phi\varphi}).
\label{eqn:tilde.ki-sl}
\end{eqnarray}
Clearly for all $y > 0$,
\begin{equation}\label{eqn:tilde.k.upper}
K(\bar{b},y,\varphi,k) \leq \tilde{K}(\bar{b},y,\varphi,k), \hspace{1em} \mbox{ for all } b \in \sqrt{k} \cdot R,\ \varphi \in A \backslash \{\phi\}.
\end{equation}
Fix a sequence of $\{\delta_\phi, \phi \in A\} \subset (0,1/2)$
such that $w^\ast_\phi < w^\ast_\varphi$ if and only if
$\delta_\phi < \delta_\varphi$ for all pairs $\phi,\varphi \in A$.
As we shall see, $\delta_\phi$ acts as an error bound for $L^k((b,\phi))$.
We are now in a position to define
\begin{eqnarray}
\underline{K}(\bar{b},y,\varphi,k) &:= &
\begin{cases}
\sum_{Q \subseteq A \backslash \{\varphi\}} 
\mathcal{P}_{\varphi \cup Q} 
(\tau_b \circ T^{y\sqrt{k},\min(y\sqrt{k}, w^\ast_\phi\sqrt{k} - k^{\delta_\phi})}_{\phi\varphi}) 
& \text{if  $w^\ast_\phi \leq w^\ast_\varphi$},
\label{eqn:under.equal} \\
\sum_{Q \subseteq A \backslash \{\varphi\}} 
\mathcal{P}_{\varphi \cup Q} (\tau_b \circ T^{y\sqrt{k},\min(y\sqrt{k}, w^\ast_\varphi\sqrt{k} - k^{\delta_\varphi})}_{\phi\varphi})
& \text{ if $w^\ast_\phi > w^\ast_\varphi$}, \label{eqn:under.above} 
\end{cases}\\
\overline{K}(\bar{b},y,\varphi,k) &:=&
\begin{cases}
\sum_{Q \subseteq A \backslash \{\varphi\}} \mathcal{P}_{\varphi \cup Q} (\tau_b \circ T^{y\sqrt{k},\min(y\sqrt{k}, w^\ast_\phi\sqrt{k} + k^{\delta_\phi})}_{\phi\varphi}) & \text{ if  $w^\ast_\phi \leq w^\ast_\varphi$}, \label{eqn:over.equal} \\
\sum_{Q \subseteq A \backslash \{\varphi\}} \mathcal{P}_{\varphi \cup Q} (\tau_b \circ T^{y\sqrt{k},\min(y\sqrt{k}, w^\ast_\varphi\sqrt{k} + k^{\delta_\varphi})}_{\phi\varphi}) & \text{ if  $w^\ast_\phi > w^\ast_\varphi$}. \label{eqn:over.above}
\end{cases}
\end{eqnarray}
These definitions are illustrated by Fig. \ref{fig:cutoff}.
Note that all points counted in these two definitions are potential killers of $\bar b$ for distance $y$.
We now show that these functions satisfy (\ref{eqn:sandwich}).

\begin{figure}[h]
\includegraphics{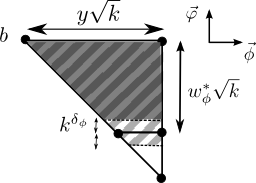}
\caption{Definition (\ref{eqn:under.equal}) illustrated for $w^\ast_\phi \leq w^\ast_\varphi$. We claim that with high probability, in $\G^k$, all motorcycles $\bar{a} = (a,\varphi)$ with $a$ in the gray region will cross the line $b + [0,y\sqrt{k}]$, and all such motorcycles with $a$ in the white region will not. We define $\underline{K}(b,y,\varphi)$ to be the number of sites of the PPP in the gray region which spawn motorcycles that travel in direction $\varphi$, and $\overline{K}(b,y,\varphi)$ to be that number for the striped region. The case $w^\ast_\phi > w^\ast_\varphi$ is the same figure with $w^\ast_\varphi\sqrt{k}$ replacing $w^\ast_\phi\sqrt{k}$.}
\label{fig:cutoff}
\end{figure}

\begin{lem}\label{lem:keybound}
Fix $k \in \mathbb{N}$, a motorcycle $\bar{b} = (b,\phi)$ with $b \in \sqrt{k} \cdot R$, angle $\varphi \in A, \varphi \neq \phi$. 
For fixed sequences $\underline{y} = \underline{y}(k) = w^\ast_\phi + o(1) > w^\ast_\phi$, $\overline{y} = \overline{y}(k) = w^\ast_\phi - o(1) < w^\ast_\phi$, as $k \to \infty$,
\begin{align}
\underline{K}(\bar{b},\underline{y},\varphi,k) \leq K(\bar{b},\underline{y},\varphi,k) & \mbox{ for all $b \in \sqrt{k} \cdot R$ w.h.p.,} \label{eqn:sandwich.under}\\
K(\bar{b},\overline{y},\varphi,k) \leq \overline{K}(\bar{b},\overline{y},\varphi,k)& \mbox{ for all $b \in \sqrt{k} \cdot R$ w.h.p..} \label{eqn:sandwich.over}
\end{align}
In particular, for $\underline{y} := w_\phi^\ast + k^{\delta_\phi-1/2}$ and $\overline{y} = w_\phi^\ast - k^{\delta_\phi-1/2}$, then
\begin{align}
\underline{K}(\bar{b},\underline{y},k) \geq k &\hspace{1em} \mbox{for all $b \in \sqrt{k} \cdot R$, w.h.p.,} \label{eqn:under.y} \\
\overline{K}(\bar{b},\overline{y},k) \leq k &\hspace{1em} \mbox{for all $b \in \sqrt{k} \cdot R$, w.h.p..} \label{eqn:over.y}
\end{align}
\end{lem}
Order the angles in $A$ as $\phi_1, \ldots, \phi_{M'}$ such that the corresponding values $w^\ast_\phi$ are non-decreasing, that is, if $i < j$, then $w^\ast_i = w^\ast_{\phi_i} \leq w^\ast_j = w^\ast_{\phi_j}$. We shall prove Lemma \ref{lem:keybound} by an inductive argument along $i = 1, \ldots, |A|$. Each step in the induction involves fixing $i$, and establishing (\ref{eqn:sandwich.under}) and (\ref{eqn:sandwich.over}) for pairs $(\phi,\varphi) = (\phi_i,\phi_j)$ and $(\phi,\varphi) = (\phi_j,\phi_i)$, for all $j = 1, \ldots |A|$ such that $j > i$, and then use this and the induction hypothesis to prove (\ref{eqn:under.y}) and (\ref{eqn:over.y}) for $\phi = \phi_i$. 
For clarity, we restate each case of Lemma \ref{lem:keybound} in the induction proof as a lemma in itself.

\begin{lem}\label{lem:base.case.1}
For all $j > 1$,
Equations (\ref{eqn:sandwich.under}) and (\ref{eqn:sandwich.over}) hold for
$\phi = \phi_1$ and $\varphi = \phi_j$. 
\end{lem}
\begin{proof}
By definition, $w^\ast_\phi = w^\ast_1$ and $w^\ast_\varphi = w^\ast_j$. As $w^\ast_1 \leq w^\ast_j$ and $\overline{y} \leq w^\ast_1$, by the definition of $\overline{K}$ in (\ref{eqn:over.equal}),
$$\overline{K}(\bar{b},\overline{y},\phi_j,k) = \sum_{Q \subseteq A \backslash \{\phi_j\}} \mathcal{P}_{\phi_j \cup Q} (\tau_b \circ T^{\overline{y}\sqrt{k},\overline{y}\sqrt{k}}_{\phi_1\phi_j}) = \tilde{K}(\bar{b},\overline{y},\phi_j,k). $$
So $\overline{K}(\bar{b},\overline{y},\phi_j,k)$
is trivially an upper-bound for $K(\bar{b},\overline{y},\varphi,k)$ by (\ref{eqn:tilde.k.upper}). This proves (\ref{eqn:sandwich.over}). 

Now consider the lower bound (\ref{eqn:sandwich.under}).  Since $\underline{y} \geq w^\ast_1$, by the definition of $\underline{K}$ in (\ref{eqn:under.above}), 
$$
\underline{K}(\bar{b},\underline{y},\phi_j,k) = \sum_{Q \subseteq A \backslash \{\phi_j\}} \mathcal{P}_{\phi_j \cup Q} (\tau_b \circ T^{\underline{y}\sqrt{k},w^\ast_1\sqrt{k} - k^{\delta_1}}_{\phi_1\phi_j}).$$
For all potential killers $\bar{a} = (a,\varphi)$
of $\bar b$ for $\underline y$,
define $d(\bar{a},\bar{b})$ to be the positive real number
such that $d(\bar{a},\bar{b})\sqrt{k}$ is the distance
$\bar{a}$ has to travel to meet the path of $\bar{b}$.
In order to prove (\ref{eqn:sandwich.under}),
it is enough to show that all potential killers
$\bar{a} = (a,\phi_j)$ with 
$d(\bar{a},\bar{b})\sqrt{k} < w^\ast_1\sqrt{k} - k^{\delta_1}$
are would-be killers of $\bar b$ with high probability.
Fix such a motorcycle~$\bar{a}$.
Until its supposed meeting with $\bar{b}$, $\bar{a}$
can only be killed at most $\tilde{K}(\bar{a},d(\bar{a},\bar{b}),k)$ times.
So it is enough to show that
\begin{equation}\label{eqn:keybound}
\sup_{b \in \sqrt{k} \cdot R}\sup_{a \in \tau_b \circ T^{\underline{y}\sqrt{k},w^\ast_1\sqrt{k}-k^{\delta_1}}_{\phi_1\phi_j}} \tilde{K}(\bar{a},d(\bar{a},\bar{b}),k) \leq k \hspace{1em} \mbox{w.h.p.}
\end{equation}
From (\ref{eqn:tilde.ki-sl}), for all $z > 0$, 
\begin{equation}\label{eqn:bar.k}
\tilde{K}(\bar{a},z,k) = 
\sum_{Q \subseteq A \backslash \{\phi_j\}}
\sum_{\psi \in Q}
(\mathcal{P}_Q +\mathcal{P}_{Q\cup\phi_j})(\tau_a \circ T^{\sqrt{k}z,\sqrt{k}z}_{\phi_j\psi}). 
\end{equation}
Set $R_k' := \sqrt{k} \cdot R \oplus \tau_0 \circ T^{\underline{y}\sqrt{k},w_1^\ast\sqrt{k}-k^{\delta_1}}_{\phi_1\phi_j},$
where $\oplus$ denotes the Minkowski sum. We have
\begin{align}
& \sup_{b \in \sqrt{k} \cdot R}\sup_{a \in \tau_b \circ T^{\underline{y}\sqrt{k},w_1^\ast\sqrt{k}-k^{\delta_1}}_{\phi_1\phi_j}} \tilde{K}(\bar{a},d(\bar{a},\bar{b}),k) \nonumber\\
=& \sup_{a \in R_k'}\sum_{Q \subseteq A \backslash \{\phi_j\}} \sum_{\psi \in Q}(\mathcal{P}_Q +\mathcal{P}_{Q\cup\phi_j})(\tau_a \circ T^{d(\bar{a},\bar{b})\sqrt{k},d(\bar{a},\bar{b})\sqrt{k}}_{\phi_j\psi}) & \mbox{by (\ref{eqn:bar.k})} \nonumber\\ 
< & \sup_{a \in R_k'}\sum_{Q \subseteq A \backslash \{\phi_j\}} \sum_{\psi \in Q} (\mathcal{P}_Q + \mathcal{P}_{Q\cup\phi_j})(\tau_a \circ T^{w^\ast_1\sqrt{k} - k^{\delta_1}, w^\ast_1\sqrt{k} - k^{\delta_1}}_{\phi_j\psi}) & \mbox{as }d(\bar{a},\bar{b})\sqrt{k} \leq w^\ast_1\sqrt{k} - k^{\delta_1} \nonumber\\
\leq&  \sum_{Q \subseteq A \backslash \{\phi_j\}} \sum_{\psi \in Q}\sup_{a \in R_k'} (\mathcal{P}_Q+\mathcal{P}_{Q\cup\phi_j}) (\tau_a \circ T^{ w^\ast_1\sqrt{k} - k^{\delta_1}, w^\ast_1\sqrt{k} - k^{\delta_1}}_{\phi_j\psi}) \label{eqn:rkprime}
\end{align}
Note that 
$$\tau_0 \circ T^{\underline{y}\sqrt{k},w_1^\ast\sqrt{k}-k^{\delta_1}}_{\phi_1\phi_j} \subset \tau_0 \circ \sqrt{k} \cdot T^{\underline{y},\underline{y}}_{\phi_1\phi_j} \subset \tau_0 \circ T^{2w^\ast_1,2w^\ast_1}_{\phi_1\phi_j},$$ 
so $R_k' \subset \sqrt{k} \cdot R' \subset \R^2$ for the compact set $R' := R \oplus \tau_0 \circ T^{2w^\ast_1,2w^\ast_1}_{\phi_1\phi_j} \subset \R^2$, independent of $k$. Thus (\ref{eqn:rkprime}) is upper bounded by
\begin{equation}\label{eqn:rkprime2}
\sum_{Q \subseteq A \backslash \{\phi_j\}} \sum_{\psi \in Q}\sup_{a \in \sqrt{k} \cdot R'} (\mathcal{P}_Q+\mathcal{P}_{Q\cup\phi_j}) (\tau_a \circ T^{ w^\ast_1\sqrt{k} - k^{\delta_1}, w^\ast_1\sqrt{k} - k^{\delta_1}}_{\phi_j\psi}).
\end{equation}
As the compact set $R'$ is independent of $k$, one can apply Lemma \ref{lem:sup.bound} to (\ref{eqn:rkprime2}). Union bound over the summands tells us that the sum in (\ref{eqn:rkprime2}) is not far from the mean, which is
$$ 
\sum_{Q \subseteq A \backslash \{\phi_j\}} \sum_{\psi \in Q}\E
(\mathcal{P}_Q+\mathcal{P}_{Q\cup\phi_j}) (\tau_a \circ T^{ w^\ast_1\sqrt{k} - k^{\delta_1}, w^\ast_1\sqrt{k} - k^{\delta_1}}_{\phi_j\psi}) = (w^\ast_1\sqrt{k}-k^{\delta_1})^2 \mathcal{E}(\mathbf{1}, \phi_j).
$$
Explicitly, for each $Q \subseteq A \backslash \{\phi_j\}$,
and each $\psi \in Q$, apply Lemma \ref{lem:sup.bound} with
the sets $R = R'$, $S = T^{w^\ast_1-k^{\delta_1-1/2},w^\ast_1-k^{\delta_1-1/2}}_{\phi_j\psi}$ and the Poisson point processes $k\cdot\mathcal{P}_Q$ and
$k\cdot \mathcal{P}_{Q\cup\phi_j}$ to obtain
$$\sup_{a \in R'}
(\mathcal{P}_Q+\mathcal{P}_{Q\cap\phi_j}) (\tau_b \circ T^{w^\ast_1\sqrt{k}-k^{\delta_1},w^\ast_1\sqrt{k}-k^{\delta_1}}_{\phi_j\psi}) \leq k (\mu_Q+\mu_{Q\cap \phi_j})|S| + O(|S|k^{1/2+\epsilon}) \mbox{ w.h.p.}  $$
Note that
$|S| = O((w^\ast_1-k^{{\delta_1}-1/2})^2) \leq O((w^\ast_1)^2) = O(1),$ 
so $O(|S|k^{1/2+\epsilon}) = O(k^{1/2+\epsilon})$. 
Fix $0 < \epsilon < \delta_\phi$. Taking a union bound over all pairs $(Q, \psi)$, we get
$$\sum_{Q \subseteq A \backslash \{\phi_j\}} \sum_{\psi \in Q}
\sup_{a \in \sqrt{k}\cdot R'}
(\mathcal{P}_Q+\mathcal{P}_{Q\cup\phi_j}) (\tau_a \circ T^{w^\ast_1\sqrt{k}-k^{\delta_1},w^\ast_1\sqrt{k}-k^{\delta_1}}_{\phi_j\psi})
= (w^\ast_1\sqrt{k}-k^{\delta_1})^2 \mathcal{E}(\mathbf{1}, \phi_j) + O(k^{1/2+\epsilon}).$$
Finally, by (\ref{eqn:wstar1}), for some constant $C > 0$,
with high probability
$$(w^\ast_1\sqrt{k}-k^{\delta_1})^2 \mathcal{E}(\mathbf{1}, \phi_j) + O(k^{1/2+\epsilon})
\leq \left(\frac{w^\ast_1\sqrt{k} -k^{\delta_1}}{w^\ast_1}\right)^2 + O(k^{1/2+\epsilon}) \leq k - Ck^{1/2+{\delta_1}} \leq k, $$
for $k$ large enough.
This establishes (\ref{eqn:sandwich.under}) for the case $\phi = \phi_1$, $\varphi = \phi_j$, $j > 1$. 
\end{proof}

\begin{lem}\label{lem:base.case.2}
Equations (\ref{eqn:under.y}) and (\ref{eqn:over.y}) hold for $\phi = \phi_1$. 
\end{lem}

\begin{proof}
Define the weight vector $w = \bar{y}\sqrt{k} \cdot \mathbf{1}$. 
For each $Q \subseteq A \backslash \{\phi_1\}$,
set $S = T^w_{\phi_1Q}$, and for each $\phi_j \in Q$,
apply Lemma \ref{lem:sup.bound} with sets $R, S$, and the Poisson point process
$\mathcal{P}_{\phi_1\phi_j}$.
As before, note that in each case $S$, $|S| = O(1)$.
By the union bound over all such sets $Q$ and such angles $\phi_j$, 
\begin{align}
\sup_{b \in \sqrt{k} \cdot R} \overline{K}(\bar{b},\overline{y},k) 
&\leq \sum_{Q \subseteq A \backslash \{\phi_1\}} \sum_{\varphi \in Q}\sup_{b \in \sqrt{k} \cdot R}|(\mathcal{P}_Q + \mathcal{P}_{Q\cup\phi_1}) \cap(\tau_b \circ T^w_{\phi_1Q})| &\mbox{ by (\ref{eqn:sandwich.over})} \nonumber \\
&\leq \mathcal{E}(w, \phi_1) +O(|S|k^{1/2 + \epsilon}) &\mbox{ by Lemma \ref{lem:sup.bound}} \nonumber\\
&= \left(\frac{w_1^\ast\sqrt{k} - k^{\delta_{\phi_1}}}{w_1^\ast}\right)^2+
O(k^{1/2 + \epsilon}) &\mbox{ by definition of }\bar{y}\nonumber \\
& \leq k - Ck^{1/2+{\delta_{\phi_1}}}\leq k
\end{align}
with high probability, for large enough $k$, and for some constant $C > 0$. 
This proves (\ref{eqn:over.y}) for $\phi = \phi_1$. 

Similarly, define the weight vector $w$ with
$$ w_1 = \underline{y} = w^\ast_1 + k^{\delta_1-1/2}, \hspace{1em} w_j = w^\ast_1 - k^{\delta_1-1/2} \mbox{ for } j > 1.$$
By the same argument as above, from (\ref{eqn:sandwich.under}) for $\phi = \phi_1$
and Lemma \ref{lem:sup.bound}, we have
$$
\inf_{b \in \sqrt{k} \cdot R} \underline{K}(\bar{b},\underline{y},k) \geq  \mathcal{E}(\sqrt{k} \cdot w, \phi_1) + O(k^{1/2 + \epsilon}).$$
By the definition of $\mathcal{E}(\sqrt{k}\cdot w, \phi_1)$ given in (\ref{eqn:defn.e}),
and the geometry of the regions $T^{\sqrt{k} \cdot w}_{\phi Q}$, we have
$$ \mathcal{E}(\sqrt{k} \cdot w, \phi_1) = \mathcal{E}((w^\ast_1\sqrt{k}+k^{\delta_1}) \cdot \mathbf{1}, 1) - O(k^{2\delta_1}). $$
Therefore, for some constant $C > 0$,
$$ \inf_{b \in \sqrt{k} \cdot R} \underline{K}(\bar{b},\underline{y},k) \geq
\left(\frac{w_1^\ast\sqrt{k} +k^{\delta_1}}{w_1^\ast}\right)^2 
- O(k^{2\delta_1}) + O(k^{1/2 + \epsilon}) = k + Ck^{1/2+\delta_1} \geq k $$
with high probability for large enough $k$. This proves (\ref{eqn:under.y}) for $\phi = \phi_1$.
\end{proof}

\begin{lem}\label{lem:base.case.3}
For all $j > 1$,
Equations (\ref{eqn:sandwich.under}) and (\ref{eqn:sandwich.over})
hold for $\varphi = \phi_1$, $\phi = \phi_j$.
\end{lem}
\begin{proof}
If $w^\ast_\phi = w^\ast_j = w^\ast_\varphi = w^\ast_1$, then the previous argument applies.
So we only need to consider $\phi = \phi_j$ such that $w^\ast_j > w^\ast_1$. In this case, for large enough $k$,
$$\underline{K}(\bar{b},\underline{y},\phi_1,k) = \sum_{Q \subseteq A \backslash \{\phi_1\}} |\mathcal{P}_{\phi_1 \cup Q}\cap (\tau_b \circ T^{\underline{y}\sqrt{k},w^\ast_1\sqrt{k} - k^{\delta_1}}_{\phi_j\phi_1})|
$$
and
$$\overline{K}(\bar{b},\overline{y},\phi_1,k) = \sum_{Q \subseteq A \backslash \{\phi_1\}} |\mathcal{P}_{\phi_1 \cup Q} \cap (\tau_b \circ T^{\overline{y}\sqrt{k},w^\ast_1\sqrt{k} + k^{\delta_1}}_{\phi_j\phi_1})|, $$
where here 
$\overline y= w^*_{j}-o(1)$ and
$\underline y= w^*_{j}+o(1)$.
For the lower bound, we 
need to show that potential killers $\bar{a} = (a,\phi_1)$ with $d(\bar{a},\bar{b})\sqrt{k} < w^\ast_1\sqrt{k}-k^{\delta_1}$ are would-be killers. Let $R' := R \oplus \tau_0 \circ T^{\overline{y},w^\ast_1-k^{\delta_1-1/2}}_{\phi_j\phi_1}$. By (\ref{eqn:sandwich.over}) and (\ref{eqn:over.y}) for the case $\bar{a} = (a,\phi_1)$
proved in Lemmas \ref{lem:base.case.1} and \ref{lem:base.case.2}, we have
\begin{align*}
\sup_{b \in \sqrt{k} \cdot R}\sup_{a \in \tau_b \circ T^{\overline{y}\sqrt{k},w^\ast_1\sqrt{k}-k^{\delta_1}}_{\phi_j\phi_1}} K(\bar{a},d(\bar{a},\bar{b}),k) \leq \sup_{a \in \sqrt{k} \cdot R'} \overline{K}(\bar{a},w^\ast_1-k^{{\delta_1}-1/2},k) \leq k,
\end{align*}
so this is the desired result. Similarly, for the upper bound, we need to show that potential killers with $d(\bar{a},\bar{b})\sqrt{k} > w^\ast_1\sqrt{k}+k^{\delta_1}$ cannot meet $\bar{b}$ in $\G^k$. Apply (\ref{eqn:sandwich.under}) and (\ref{eqn:under.y}) for the case $\bar{a} = (a,\phi_1)$ proved in Lemmas \ref{lem:base.case.1} and \ref{lem:base.case.2}, we have
$$
\inf_{b \in \sqrt{k} \cdot R}\inf_{a \in \tau_b \circ \left(T^{\underline{y}\sqrt{k},\underline{y}\sqrt{k}}_{\phi_j\phi_1} \backslash T^{\underline{y}\sqrt{k},w^\ast_1\sqrt{k}+k^{\delta_1}}_{\phi_j\phi_1}\right)} K(\bar{a},d(\bar{a},\bar{b}),k) \geq
\inf_{a \in \sqrt{k} \cdot R'} \underline{K}(\bar{a},w^\ast_1+k^{{\delta_1}-1/2},k) \geq k \hspace{1em}
\mbox{w.h.p.},
$$
where $R' = R + \tau_0 \circ \left(T^{\underline{y},\underline{y}}_{\phi_j\phi_1} \backslash T^{\underline{y},w^\ast_1\sqrt{k}+k^{\delta_1}}_{\phi_j\phi_1}\right)$. This completes the proof.
\end{proof}

Suppose Lemma \ref{lem:keybound} holds for all $\phi_1, \ldots, \phi_{i-1}$. 
This means we have proven that (\ref{eqn:under.y}) and (\ref{eqn:over.y})
hold for $\bar{b} = (b,\phi_\ell)$ for all $\ell = 1, \ldots, i-1$,
and that (\ref{eqn:sandwich.under}) and (\ref{eqn:sandwich.over})
hold for pairs $(\phi,\varphi) = (\phi_\ell,\phi_j)$ and $(\phi,\varphi)=(\phi_j,\phi_\ell)$,
for all $\ell = 1, \ldots, i-1$, and $j > \ell$. 
Thus, we may assume that $w_{i-1}^\ast < w_i^\ast$. 
Lemmas \ref{lem:base.case.1}, \ref{lem:base.case.2} and \ref{lem:base.case.3} establish the base case with $\phi = \phi_1$. We now prove the various induction statements for $\phi = \phi_i$.

\begin{lem}
Equations (\ref{eqn:sandwich.under}) and (\ref{eqn:sandwich.over}) hold for $\phi = \phi_i$, $\varphi = \phi_j$, $j \neq i$.
\end{lem}
\begin{proof}
The induction hypothesis already covers the case $j < i$.
Consider the case $j > i$.
Again, the upper bound is equal to $\tilde{K}(\bar{b},y,\phi_j,k)$,
so we only need to establish the lower bound.
From the induction assumptions,
by the same consideration as in Lemma \ref{lem:base.case.1}, it is enough to show that
\begin{equation}\label{eqn:sandwich.under.induction}
\sup_{b \in \sqrt{k} \cdot R}\sup_{a \in \tau_b \circ T^{y\sqrt{k},,w^\ast_{\phi_i}\sqrt{k}-k^{\delta_{\phi_i}}}_{\phi_i\phi_j}} \tilde{K}^{i,j}(\bar{a},d(\bar{a},\bar{b}),k) \leq k \hspace{1em} w.h.p.,
\end{equation}
where
\begin{equation}\label{eqn:tilde.k.j}
\tilde{K}^{i,j}(\bar{a},z,k) := \sum_{\ell=1}^{i-1}\overline{K}(\bar{a},z,\phi_\ell,k)
+ \sum_{\ell=i, \ell \neq j}^{|A|} \tilde{K}(\bar{a},z,\phi_\ell,k). 
\end{equation}
With $R'$ defined analogous to the base case, we have
\begin{align*}
& \sup_{b \in \sqrt{k} \cdot R}\sup_{a \in \tau_b \circ T^{y\sqrt{k},w^\ast_{\phi_i}\sqrt{k}-k^{\delta_{\phi_i}}}_{\phi_i\phi_j}} \tilde{K}^{i,j}(\bar{a},d(\bar{a},\bar{b}),k) \leq 
\sup_{a \in \sqrt{k}\cdot R'} \tilde{K}^{i,j}(\bar{a},w^\ast_{\phi_i}-k^{\delta_{\phi_i}-1/2},k) \\
&\leq \sup_{a \in \sqrt{k}\cdot R'} \sum_{\ell=1}^{i-1}\overline{K}(\bar{a},w^\ast_{\phi_i}-k^{\delta_{\phi_i}-1/2},\phi_\ell,k) + \sup_{a \in \sqrt{k}\cdot R'} \sum_{\ell=i, \ell \neq j}^{|A|} \tilde{K}(\bar{a},w^\ast_{\phi_i}-k^{\delta_{\phi_i}-1/2},\phi_\ell,k) \\
&\leq \sum_{\ell=1}^{i-1}\sum_{Q \subseteq A \backslash \{\phi_\ell\}}
\E\mathcal{P}_{\phi_\ell \cup Q} (\tau_a \circ T^{w^\ast_{\phi_i}\sqrt{k}-k^{\delta_{\phi_i}},w^\ast_{\phi_\ell}\sqrt{k} + k^{\delta_{\phi_\ell}}}_{\phi_j\phi_\ell}) \\
& + \sum_{\ell=i,\ell \neq j}^{|A|}\sum_{Q \subseteq A \backslash \{\phi_\ell\}} 
\E\mathcal{P}_{\phi_\ell \cup Q} (\tau_a \circ T^{w^\ast_{\phi_i}\sqrt{k} - k^{\delta_{\phi_i}},w^\ast_{\phi_i}\sqrt{k} - k^{\delta_{\phi_i}}}_{\phi_j\phi_\ell}) + O(k^{1/2+\epsilon}) \hspace{3em} \mbox{ by Lemma \ref{lem:sup.bound}},
\end{align*}
where $\epsilon > 0$ is an arbitrarily small constant.
Define the weight vector $w \in \R^M$ with
$$ w_\ell = w^\ast_\ell + k^{\delta_\ell-1/2} \mbox{ for } \ell = 1, \ldots, i-1, \hspace{1em} w_\ell = w^\ast_i - k^{\delta_i-1/2} \mbox{ for } \ell \geq i. $$
Then the sum of the expectations in the last expression
is equal to $\mathcal{E}(\sqrt{k}\cdot w, \phi_j)$. Define $w' \in \R^M$ via
$$ w'_\ell = w^\ast_\ell, \mbox{ for } \ell = 1, \ldots, i-1,
\hspace{1em} w'_\ell = w^\ast_i, \mbox{ for } \ell \geq i. $$
By Lemma \ref{lem:function.e}, we have 
$$ \mathcal{E}(\sqrt{k}\cdot w, \phi_j) = \mathcal{E}(\sqrt{k} \cdot w', \phi_j) + o(k^{1/2+\delta_i}) - Ck^{1/2+\delta_i} $$
for some constant $C > 0$.
Now, by definition of $w^\ast$, 
$$ \mathcal{E}(\sqrt{k} \cdot w', \phi_j) = k.  $$
Therefore, the quantity we need is upper-bounded by 
$$
k + o(k^{1/2+\delta_i}) - Ck^{1/2+\delta_i} + O(k^{1/2+\epsilon})= k - Ck^{1/2+\delta_{\phi_i}} \leq k $$
with high probability for large enough $k$. This proves (\ref{eqn:sandwich.under.induction}), as needed.
\end{proof}
\begin{lem}
Equations (\ref{eqn:under.y}) and (\ref{eqn:over.y}) hold for $\phi = \phi_i$.
\end{lem}
\begin{proof}
Define the weight vector $w \in \R^M$ by
$$ w_\ell = w_\ell^\ast + k^{\delta_\ell-1/2} \mbox{ if } \ell \leq i,\qquad 
w_\ell = \bar{y} = w_i^\ast - k^{\delta_i-1/2} \mbox{ if } \ell > i. $$
By the same argument as in Lemma \ref{lem:base.case.2}, 
with high probability,
\begin{align}
\sup_{b \in \sqrt{k} \cdot R} \overline{K}(\bar{b},\overline{y},k) 
&\leq \sum_{Q \subseteq A \backslash \{\phi_i\}} \sum_{\varphi \in Q}\sup_{b \in \sqrt{k} \cdot R}
(\mathcal{P}_Q + \mathcal{P}_{Q\cup \phi_i}) (\tau_b \circ T^{\sqrt{k} \cdot w}_{\phi_iQ}) &\mbox{ by (\ref{eqn:sandwich.over}) for $\phi = \phi_i$} \nonumber \\
&\leq \mathcal{E}(\sqrt{k} \cdot w, \phi) +O(k^{1/2 + \epsilon}) &\mbox{ by Lemma \ref{lem:sup.bound}} \label{eqn:bylemma13}. 
\end{align}
By the induction hypothesis, we have $\delta_i > \delta_\ell$ for $\ell < i$, $w_i^\ast > w_{i-1}^\ast$. By Lemma \ref{lem:function.e}, we have
$$ \mathcal{E}(\sqrt{k} \cdot w, \phi_i) = k + o(k^{1/2+\delta_i}) - Ck^{1/2+\delta_i}$$
for some constant $C > 0$. 
As $\epsilon > 0$ in (\ref{eqn:bylemma13}) is an arbitrary constant,
$$
\sup_{b \in \sqrt{k} \cdot R} \overline{K}(\bar{b},\overline{y},k) \leq k + o(k^{1/2+\delta_i}) + O(k^{1/2+\epsilon}) - Ck^{1/2+\delta_i}\leq k - \frac{C}{2}k^{1/2+\delta_i} \leq k
$$
with high probability, large enough $k$. This proves (\ref{eqn:over.y}) for $\phi = \phi_i$.  

Similarly, define the weight vector $w \in \R^M$ with
$$ w_\ell = w^\ast_\ell - k^{\delta_\ell - 1/2} \mbox{ if } \ell < i, \hspace{1em}  w_\ell = w^\ast_i + k^{\delta_i - 1/2} \mbox{ if } \ell \geq i. $$
By the same argument as above, with (\ref{eqn:sandwich.under}) for $\phi = \phi_i$ and Lemma \ref{lem:sup.bound}, we have
$$
\inf_{b \in \sqrt{k} \cdot R} \underline{K}(\bar{b},\underline{y}) \geq  \mathcal{E}(\sqrt{k} \cdot w, \phi_i) + O(k^{1/2 + \epsilon}).$$
By Lemma \ref{lem:function.e},
$$ \mathcal{E}(\sqrt{k} \cdot w, \phi_i) = k - o(k^{1/2+\delta_i}) + Ck^{1/2+\delta_i}$$
for some constant $C > 0$. As $\epsilon > 0$ is an arbitrarily small constant,
$$ \inf_{b \in \sqrt{k} \cdot R} \underline{K}(\bar{b},\underline{y}) \geq k + Ck^{1/2+\delta_i} - o(k^{1/2+\delta_i}) + O(k^{1/2+\epsilon}) \geq k + \frac{C}{2}k^{1/2+\delta_i} \geq k $$
with high probability. This proves (\ref{eqn:under.y}) for $\phi = \phi_i$.
\end{proof}

\begin{lem}
Equations (\ref{eqn:sandwich.under}) and (\ref{eqn:sandwich.over}) hold for $\varphi = \phi_i$, $\phi = \phi_j$, $j > i$.
\end{lem}
\begin{proof}
The proof is identical to that of Lemma \ref{lem:base.case.3}.
\end{proof}

\begin{proof}[Proof of Proposition \ref{prop:concentrate}]
By (\ref{eqn:under.y}) and (\ref{eqn:over.y})
$$ \overline{y}\sqrt{k} = w^\ast_\phi\sqrt{k} - k^{\delta_\phi} \leq \sup_{b \in \sqrt{k} \cdot R }L^k((b,\phi)) \leq \underline{y} = w^\ast_\phi\sqrt{k} + k^{\delta_\phi}  \hspace{1em} \mbox{ w.h.p.} $$
Rearranging gives
$$\sup_{b \in \sqrt{k} \cdot R}|\frac{L^k((b,\phi))}{\sqrt{k}} - w^\ast_\phi| \leq  k^{\delta_\phi-1/2} = o(1),$$
as desired.
\end{proof}

\begin{ex}
Suppose $\G^1$ is the rectangular Gilbert tessellation studied in \cite{burridge2013full}. That is, each site of $\mathcal{P}$, there are four motorcycles that travel in directions the four directions north, south, east, west. Let $\G^k$ be the iterated Gilbert model. Then Theorem \ref{thm:main} states that $\G^k(\frac{1}{k})$ converges to the classical Poisson line process $\G^\infty$ with cylindrical measure $2\lambda \times (\delta_- \delta_|)$, 
where $\delta_-$ and $\delta_|$ are the Dirac delta measures at the points $0$ and $\pi/2$ on the unit circle, respectively. 
\end{ex}
  
\subsection{Iterated Gilbert with initial complex}\label{subsec:expanded}

We can generalize the iterated Gilbert model by replacing the initial
sites $\mathcal{P}$ by a germ-and-grain model, where at each site
in $\mathcal{P}$, one attaches an i.i.d. random polyhedral complex,
which contains vertices at which the motorcycles start.
Let $\G^0$ be the union of these initial polyhedral complexes,
called the initial complex. 

The general iterated Gilbert model $\G^k$ starting with $\G^0$ 
features, for each initial polyhedral complex, a collection of motorcycles
starting at some points of the complex. It is assumed that, for
each given polyhedral complex, an arm starting from this complex
never crosses the complex in questions again, nor any other
different arm of the polyhedral complex in question.
Each such motorcycle starts with a capital of $k$
lives and looses one live when
it crosses either another body of $\G^0$ or the path of a motorcycle
emanating from another body of $\G^0$. 
For a compact set $V \subset \R^2$,
define its \term{radius} to be the radius of the smallest ball 
containing $V$. If the radius of the polyhedral complex at each site
is at most $r > 0$ for some constant $r$, and has finitely many facets,
edges and vertices, then one can show that for each $k=1,2,\ldots$,
the general iterated Gilbert model $\G^k$ starting with $\G^0$ is
still a random mosaic with finite intensity. 

Note that in the presence 
of initial complexes,
the situation where one multiplies the intensity of $\mathcal P$ by
$\frac 1 k$ and that where one rescales space by $\sqrt{k}$ do not
coincide anymore. In the former case, initial complexes are not scaled, whereas they are in the latter case. In what follows, we consider the former interpretation, namely that of a Poisson point process of centroids with intensity $\frac \lambda k$ and no rescaling of the initial complexes.

We claim that $\G^0$ does not affect the scaling limit. In particular,
if the induced angle distributions on the motorcycles satisfy
the hypotheses of Theorem \ref{thm:main}, then Theorem \ref{thm:main}
holds unchanged. To see this, first, note that as $k \to \infty$, for a compact set $W \subset \R^2$,
$\G^0(\frac{1}{k}) \cap W \to \emptyset$.
Thus, points in $\G^0$ do not appear in the limit.
Second, we claim that $\G^0$ does not affect the argument
leading to Theorem \ref{thm:main} regarding the distance
a motorcycle can travel with $k$ lives. Indeed,
fix a motorcycle $\bar{b} = (b,\phi)$, and consider its number
of would-be killers on $b + [0,y\sqrt{k}\cdot \vect{\phi}]$.
On this interval, $\bar{b}$ can now hit lines in $\G^0$ and lose more lives.
But such lines must come from polyhedral complexes whose centroids
are within Euclidean distance $r$ of the line.
The number of such centroid is Poisson with mean $O(\sqrt{k})$,
with fluctuations of order $O(k^{1/4+\epsilon'})$. 
This is well-below the fluctuations $O(k^{1/2+\epsilon})$ of 
the number of would-be killers of $\bar{b}$,
and thus our argument for Theorem \ref{thm:main}
essentially goes through unchanged. 

\begin{thm}\label{thm:main2}
Let $\G^k(\mathcal{P},\mathcal{A})$ be an iterated Gilbert mosaic with initial complex $\G^0$, whose polyhedral complex at each site of $\mathcal{P}$ has radius at most $r > 0$. Then Theorem \ref{thm:main} applies. That is, as $k \to \infty$, for any compact window $W \subset \R^2$,
$$ \G^k \left(\frac{1}{k}\right)
\cap W \to \G^\infty\cap W \mbox{ in probability}, $$
where $\G^\infty$ is a Poisson line process with cylindrical measure
$\Lambda dr \times \Theta(d\theta)$, with
$$ \Lambda=  \sum_{\phi\in A}
w^*_\phi \sum_{Q\subset A, \phi\in Q} \mu_Q
$$
and $\Theta$ the probability measure
with mass 
$$\frac 1  \Lambda w^*_\phi \sum_{Q\subset A, \phi\in Q} \mu_Q$$
at $\phi^\perp$, for all $\phi\in A$,
where $w^\ast$ is defined by (\ref{eqn:constants}).
\end{thm}

\section{Application: Poisson tropical plane curves}\label{sec:tropical}
This section provides some background on tropical geometry,
and discusses why the iterated Gilbert model is the right way
to study a process of tropical plane curves from the view point
of stochastic geometry. 

\subsection{Tropical polynomials}\label{subsec:trop-poly}
Consider the tropical min-plus algebra $(\bar{\R}, \odot, \oplus)$, where $\bar{\R} = \R \cup \{+\infty\}$, $a \odot b = a+b$, $a \oplus b = \min(a,b)$. A tropical polynomial $f$ in two variables has the form
\begin{equation}\label{eqn:f}
f(x,y) = \bigoplus_{i,j \in \mathbb{N}} c_{ij}\odot x^{\odot i} y^{\odot j} = \min_{i,j \in \mathbb{N}} (c_{ij} + ix + jy),
\end{equation}
where the coefficients $c_{ij} \in \bar{\R}$.
It is assumed that only finitely many $c_{ij}$'s
are finite (recall that $+\infty$ is the zero of $\oplus$, so that
this condition simply says that there are only finitely many non
zero-monomials).
As in classical algebra, the support of $f$ is 
$$\text{supp}(f) = \{(i,j) \in \mathbb{N}^2: c_{ij} < \infty\}.$$
The convex hull of the support of $f$ is the {\em Newton polygon of $f$}.
For each $c_{ij} < \infty$, the graph of each term
$(x,y) \mapsto c_{ij}\odot x^{\odot i} y^{\odot j}$ is
a plane in $\R^3$. Thus, the graph of $f$ is the minimum of
finitely many planes, and is piecewise affine,
see Figure \ref{fig:graphf}. The tropical zeros,
or \term{tropical variety} of $f$, denoted by $V_f$,
is the set of points $(x,y) \in \R^2$ where the minimum
in (\ref{eqn:f}) is achieved at least twice, or in
other words, points where $f$ is non-differentiable.
This definition of zeros allows many classical theorems
in algebra to carry over in the tropical setting.
For example, the Fundamental Theorem of Algebra
\cite[\S 1]{maclagan2015introduction} applies tropically,
meaning that the tropical polynomials can be factorized
into a product of affine terms based on its zeros.
A deeper result is the Fundamental Theorem of Tropical
Algebraic Geometry \cite[\S 3]{maclagan2015introduction},
which gives a correspondence between the zeros of tropical
polynomials and those of classical polynomials
when the former are obtained through tropicalization of
the latter over non-Archimedian fields. 

For $d \in \mathbb{N}$, $0 < d < \infty$, say that $f$
is \term{standard} with degree $d$ if its Newton polygon
is the triangle with vertices $(0,0)$, $(0,d)$ and $(d,0)$.
In this case, we say that $V_f$ is a standard tropical plane curve. 
Since the rest of the text is about such curves, for simplicity
we will refer to them as tropical curves. The restriction to standard
Newton polygons is fundamental to the results in this section,
since it ensures that the unbounded segments (arms) of the
tropical curve can only take on certain angles. 

\subsubsection{Tropical plane curves and duality}\label{subsec:plane.curves}

A tropical plane curve $V_f$ is a polyhedral complex. 
It is convenient to work with the dual of this complex. This is 
the regular subdivision of the negative of the Newton polygon of $f$
with lift given by the coefficients $c_{ij}$'s.
This fact holds for tropical hypersurfaces of arbitrary dimensions
\cite[Proposition 3.1.6]{maclagan2015introduction}.
For simplicity we only state the definitions for the case of plane curves.

\begin{figure}[h]
\includegraphics[width = 0.8\textwidth]{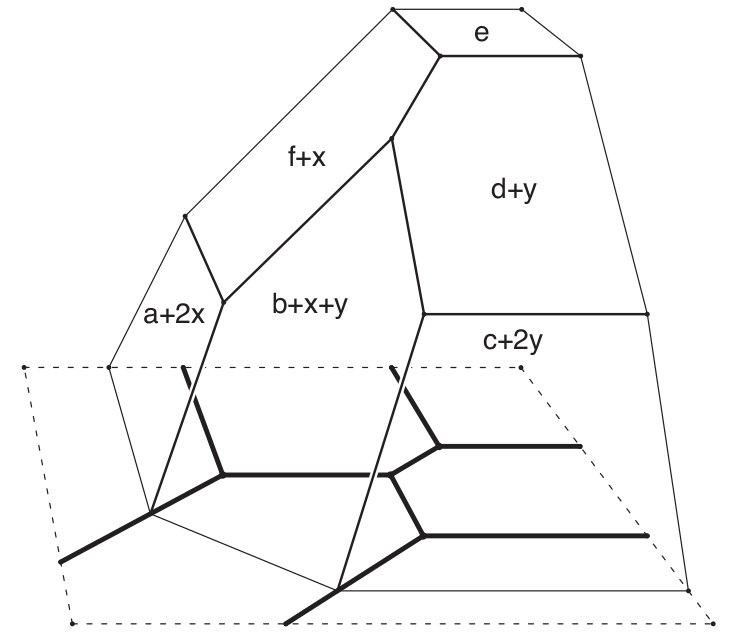}
\caption{The graph and the tropical plane curve of degree $2$
defined by $f(x,y) = a \odot x^{\odot 2} \oplus b \odot x \odot y \oplus c \odot y^{\odot 2} \oplus d \odot y \oplus e \oplus f \odot x$. Figure taken from \cite[Figure 1.3.2]{maclagan2015introduction}. The graph is the three-dimensional structure. The tropical curve is the projection of the non-differentiable points of this three-dimensional structure on the plane.}
\label{fig:graphf}
\end{figure}

We now define this regular subdivision,
which is illustrated in Figure \ref{fig:duality}.
For each integer point $(i,j)$ of the Newton polygon, one `lifts'
it up to height $c_{ij}$; one then takes the 
convex hull of the points $(i,j,c_{ij}) \in \R^3$, reflects the figure
about the origin in the $(i,j)$ plane, and finally projects the
lower faces of this convex hull back down to the plane.
This is the regular subdivision aforementioned.

To see its connection with $V_f$, define $g := -f$,
let $\hat{g}: \R^2 \to \R$ be its Legendre transform
$$ \hat{g}(u,v) = \sup_{(x,y) \in \R^2}(ux + vy - g(x,y)). $$
By a direct calculation, one finds that the graph of $\hat{g}$
is the lower convex hull of the set of points 
$\{(-(i,j),c_{ij}): (i,j) \in \text{supp}(f)\} \subset \R^3$.
Its projection onto $\R^2$ hence forms the regular subdivision
of $-\text{Newt}(f)$ with lift $c_{ij}$'s. 
By duality of the Legendre transform, one gets that $\hat{\hat{g}} = g$.
By a definition chase, one finds that this implies that the polyhedral
complex dual to the regular subdivision is precisely
the tropical plane curve defined by $f$. Figure \ref{fig:tline} gives a second illustration of this duality.

\begin{figure}[h]
\includegraphics[width = 0.8\textwidth]{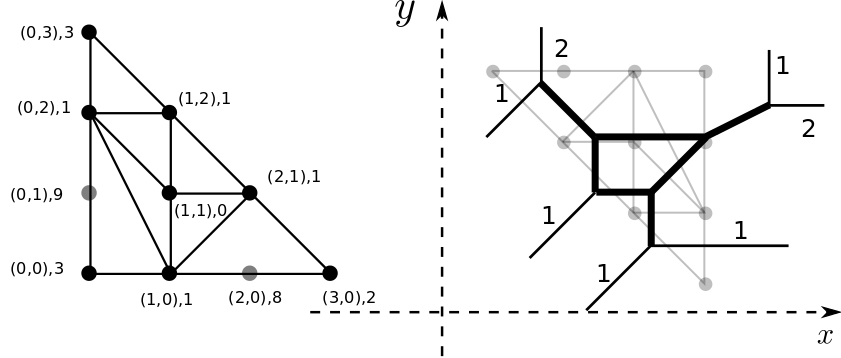}
\caption{Here $f(x,y) = 3\odot y^{\odot 3} \oplus 1\odot y^{\odot 2} \oplus 1\odot x\odot y^{\odot 2} \oplus 9\odot y \oplus x\odot y \oplus 1\odot x^{\odot 2}\odot y \oplus 3 \oplus 1\odot x \oplus 8\odot x^{\odot 2} \oplus 2\odot x^{\odot 3},$ a standard tropical curve of degree $3$. Left:
the Newton polygon $\text{Newt}(f)$ with the lift $c_{ij}$'s.
The two lattice points $(0,1)$ and $(2,0)$ are not vertices of the 
lower convex hull of the lifted points, and hence not vertices
of the regular subdivision. These two points are shown in gray. 
Right: the tropical plane curve with $-\text{Newt}(f)$ shown in gray.
The body of the curve is shown in bold.
The arms are shown in black, with numbers indicating their multiplicities.
}
\label{fig:duality}
\end{figure}

\begin{figure}[h]
\includegraphics[width=0.8\textwidth]{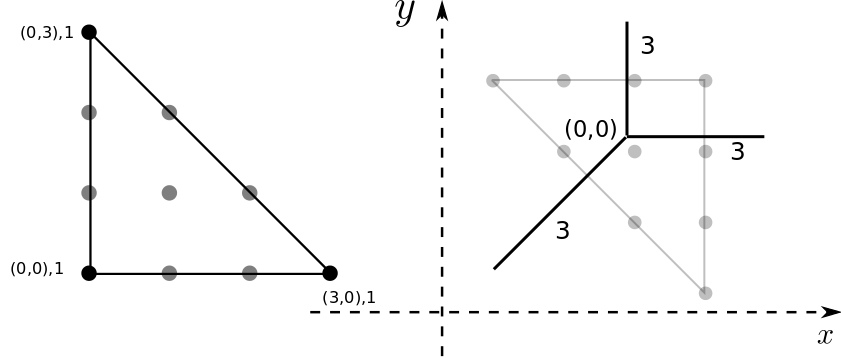}
\caption{The Newton polygon (left) and the tropical plane curve (right) defined by the polynomial $f_1(x,y) = x^{\odot 3} \oplus y^{\odot 3}\oplus 0$. As a set, this curve equals the tropical line given by $f_2(x,y) = x \oplus y \oplus 0$. However, each arm of this curve has multiplicity $3$, while each arm of the tropical line only has multiplicity $1$.} \label{fig:tline}
\end{figure}

Unlike in classical algebra, 
even up to trivial scaling of coefficients,
a tropical polynomial is not uniquely
determined by its set of zeros, or equivalently by its plane curve.
Here is however an observation that will be used later.
Let $g$ be a tropical polynomial
with $ij$-th coefficient equal to
$c_{ij}(g)$. Let $x^*$ and $y^*$ be real numbers.
Let $f$ be the tropical polynomial with $ij$-th coefficient
defined by
$$ c_{ij}(f) = \left\{ 
\begin{array}{ccc} 
c_{ij}(g) + ix^* + jy^* & \mbox{ if } & c_{ij}(g) < \infty \\
\infty & \mbox{ if } & c_{ij}(g) = \infty. 
\end{array} \right. $$
Any zero $(x,y)$ of $g$ can be mapped in
a bijective way to a zero $(\widetilde x, \widetilde y)$ of $f$
through the tropical linear transformation
$(x,y)=(x^*\odot \widetilde x, y^*\odot \widetilde y)$.
Hence, the zeros of $g$ are tropically scaled versions of
those of $f$, with the scaling coefficients determined by
$(x^*,y^*)$.

\subsubsection{Tropical plane curves as sets}
Let $f$ be a standard tropical polynomial with finite degree $d$. Its corresponding tropical curve $V_f$ is a closed set, or more precisely, a polyhedral complex in $\R^2$. It is the union of unbounded half-lines called \term{arms},
 denoted by $a(V_f)$, and a connected of set of line segments. We call the union of the later set of line segments the \term{body}, denoted by $b(V_f)$. Vertices, half-lines and line segments of $V_f$ are collectively called its \term{polyhedral facets}. The \term{multiplicity} $m(\sigma)$ of an arm or line segment $\sigma$ of $V_f$ is the lattice length of the edge of the regular subdivision of the Newton polygon of $f$ that is dual to $\sigma$. If $p \in \R^2$ is a common zero of the polynomials $f_1$ and $f_2$, formed by the intersection of polyhedral facets $\sigma_1 \subset V_{f_1}$ and $\sigma_2 \subset V_{f_2}$, then the multiplicity of $p$ is $m(p) = m(\sigma_1) m(\sigma_2)$. See \cite[\S 3]{maclagan2015introduction} for further details.

Say that an arm is \term{horizontal} (resp. \term{vertical} and \term{diagonal}) if it is parallel to the $(0,1)$ (resp. $(0,1)$ and $(1,1)$) direction, respectively. An important property of standard tropical plane curves of degree $d$ is that they have precisely $d$ arms of each of these three types. This is not necessarily true for non-standard tropical plane curves.

\begin{lem}\label{lem:arms}
Let $V_f$ be a tropical curve of degree $d$. Then an arm of $V_f$ can only have slope parallel to the $(0,1), (1,0)$ or $(1,1)$ direction. Furthermore, counting multiplicities, $V_f$ has precisely $d$ arms of each type.
\end{lem}
\begin{proof}
Let $\sigma$ be an arm or line segment of $V_f$. Then $\sigma$ is an arm of $V_f$ if and only if it is dual to an edge on the boundary of the Newton polygon of $V_f$ in its regular subdivision. Thus, an arm of $V_f$ can only have slope perpendicular to the slopes of the boundary edges of the Newton polygon of $V_f$, which are $(0,1), (1,0)$ and $(1,-1)$. This proves the first statement. For the second, note that the total multiplicities of all horizontal arms equals the lattice length of the line segment $[(0,0), (0,d)]$, which is $d$. Thus, counting multiplicities, $V_f$ has $d$ horizontal arms. The vertical and diagonal cases are proven similarly. 
\end{proof}

\begin{defn}[Centroid function]
Let $\mathcal{C}$ be the set of compact sets in $\R^2$. A centroid function $c: \mathcal{C} \to \R^2$ is 
a measurable function such that
$$ c(C + y) = y + c(C), \hspace{1em} y \in \R^2, C \in \mathcal{C}, $$
where $C + y$ is the translated set $\{x+y: x \in C\}$. 
\end{defn}

Examples of centroid functions include the center of mass of the set, or its left-most point. Since the body of a tropical curve is compact, we define the centroid of a tropical curve to be the centroid of its body. By Lemma \ref{lem:arms}, an arm of a tropical curve can therefore be represented as a mark $(a, \phi) \in \R^2 \times \{0, \pi/2, 5\pi/4\}$, where $\phi$ is the angle of its ray with respect to the $(1,0)$ vector, and $a$ is the coordinates of its apex with respect to the centroid of the curve. We can thus identify $V_f$ as a pair $(b(V_f), a(V_f))$, consisting of a compact set $b(V_f)$, its body, and a set of marks $a(V_f)$, representing its arms. 
Let $\mathcal{V} \subset \mathcal{C} \times \R^2 \times \{0, \pi/2, 5\pi/4\}$ denote the set of all such pairs of compact sets and marks which represent some tropical curve $f$. 

\subsection{A Poisson class of tropical polynomials
in two variables}

The aim of this subsection is to introduce
the Poisson based ensemble of random tropical polynomials the common
zeros of which are to be analyzed below.
This ensemble can be viewed in two ways.
The first view point is
that of the collection of the zeros of all polynomials in the
ensemble. These can be seen as a translation 
invariant collection of random sets of the Euclidean plane,
where each such set is a piecewise-linear polyhedral complex.
The second view is that of the collection of tropical polynomials
themselves. As we show below, the latter can be seen as a collection
of tropical polynomials which is invariant by all tropical scale changes.
In this sense, this collection of tropical polynomials is a fractal.
In both view points, the setting features  
$\mathcal{F}$, a distribution on standard polynomials, and
$\mathcal{P}$, a homogeneous Poisson point process on ${\mathbb R}^2$
with points $T(p)=(x(p),y(p))$, numbered
with respect to their distance to the origin.

For the first view point, we see tropical curves as compact sets with marks
and the ensemble as an instance of the classical germ-and-grain model of
stochastic geometry.
Let $V_{\mathcal{F}}$ be the
distribution induced by $\mathcal{F}$
on tropical curves.
Let $\{V_p\}_{p\in \mathbb N}$ be an i.i.d. collection of grains
sampled using $V_{\mathcal{F}}$,
To each germ $T(p)=(x(p),y(p))$ and grain $V(p)\subset {\mathbb R}^2$,
we associate
$$ W(p):= V(p) + T(p), \ p\in \mathbb N.$$ 
This collection of curves, is hence a germ grain model, and is
translation invariant by construction.

For the second view point,
let $\{f_p\}_{p\in \mathbb N}$ be an i.i.d. collection
of polynomials sampled according to $\mathcal{F}$.
As explained above, if 
the tropical polynomial
$$ f_p(x,y) 
= \min_{i,j\in \mathbb N} \left( c_{i,j}(p) + ix + jy\right)
= \bigoplus_{i,j \in \mathbb{N}} c_{ij}(p)\odot x^{\odot i} y^{\odot j}
$$
admits the plane curve $V(p)$,
then the tropical polynomial $g_p$ defined by
\begin{eqnarray*} g_p(x,y) & = &
\min_{i,j\in \mathbb N} \left( c_{i,j}(p) + i(x-x(p)) + j(y-y(p))\right)\\
& = & \bigoplus_{i,j \in \mathbb{N}} c_{ij}(p)\odot 
\left(x \odot x(p)^{\odot -1}\right)^{\odot i} \left( y \odot y(p)^{\odot -1}\right)^{\odot j}\\ & = &
f_p(x \odot x(p)^{\odot -1},y \odot y(p)^{\odot -1}),  
\end{eqnarray*}
admits the plane curve $W(p)=T(p)+V(p)$.
Here $\cdot^{\odot -1}$ denotes the inverse of 
tropical multiplication, namely $a^{\odot -1}=-a$.
The polynomials $\{g_p\}$, which form our ensemble, are hence obtained
from the i.i.d.
polynomial $\{f_p\}$ by tropical rescaling of space, where
the rescaling coefficients used for $f_p$ are $x(p)^{\odot -1}$
on the $x$ coordinate
and $y(p)^{\odot -1}$ on the $y$ coordinate, with $(x(p),y(p))$ the
coordinates of $T(p)$.

By the same argument as above, the fact that the germ-grain model
$\{W(p)\}_{p\in \mathbb N}$
is translation invariant (has a distribution which is invariant
by the translation by $t=(u,v)$ for all $t\in {\mathbb R}^2$)
can be rephrased by saying
that the family of tropical polynomials $\{g_p\}_{p\in \mathbb N}$
introduced above
is {\em scale invariant} in the tropical sense,
namely the ensemble of polynomials $\{g_p(x,y)\}_{p\in \mathbb N}$
has the same distribution as the ensemble
$\{g_p(x\otimes u^{\odot -1},y \otimes v^{\odot -1})\}_{p\in \mathbb N}$
for all $(u,v)\in {\mathbb R}^2$.

\subsection{Common zeros of the Poisson ensemble and iterated Gilbert mosaics}

Consider the germ grain ensemble defined above.
Since the germs are in general positions, any pair of plane curves
$W(p)=T(p) + V_f(T(p)), W(q)=T-q + V_f(T(q))$ will a.s. intersect
at finitely many points. These intersections, which are the common zeros to the
corresponding pair of tropical polynomials,
are of three types:
\begin{enumerate}
\item arm-arm: intersection of an arm of $W(p)$ and an arm of $W(q)$;
\item arm-body: intersection of an arm of $W(p)$ and the body of $W(q)$,
or the symmetrical situation;
\item body-body: intersection of the body of $W(p)$ and that of $W(q)$.
\end{enumerate}
As the initial curves
are stationary, intersections of each type form stationary sets.
The main of this subsection is to leverage the scaling 
law of the Gilbert model to study certain asymptotic
properties of these sets in the regime where the intensity
of $\mathcal P$ tends to 0 like $\frac \lambda k$ with $\lambda$
constant and $k$ tending to infinity.
In this regime, we will discuss
the scaling properties of the point process  
$\mathcal{I}_k$ of arm-arm intersections of order less than $k$,
and those of the set $\mathcal{J}_k$ of arm-body intersections
that a typical body has with arms of order less than $k$.

\subsubsection{Arm-arm common zeros}
In the classical (non-tropical) setting,
the intersection process of Poisson lines in ${\mathbb R}^2$
is a point process with finite intensity.
However, for plane curves of
Poisson tropical polynomials (and even for those of tropical lines),
the set $\mathcal{I}$ of arm-arm intersections
is not the support of a point process 
(cf. Proposition \ref{prop:no.union} below). 
Hence the need for a refinement
of common zeros through their order.

The iterated Gilbert model assigns to each common zero of a
pair of tropical polynomials such an order, which generally indicates
its proximity to the centroids. The variant used is that with
initial complex as considered in Section \ref{subsec:expanded}. 
Vertices of the $k$-th mosaic
$\G^k$ consist of all common zeros of order at most~$k$,
denoted by $\mathcal{I}_k$, and all vertices of $\G^0$.
The sequence $\mathcal{I}_k$ is an increasing family of
stationary sets which are supports of point 
processes, and which tend to $\mathcal{I}$ as $k$ tends to infinity.

To each marked point $(a,\phi)$ on $V_f(T)$, where $T$ is a point of 
$\mathcal P$,
introduce a motorcycle $(a,\phi)$.
Let $$\G^0 := \bigcup_{T \in \mathcal{P}} b(V_f(T))$$ 
be the initial complex consisting of the bodies of the tropical plane curves.
For $k = 1, 2, \ldots$, let $\G^k(\mathcal{F}, \mathcal{P})$
denote the $k$-th order iterated Gilbert model starting from
initial complex $\G^0$, with the given motorcycles. 

\begin{lem}
Suppose $\mathcal{F}$ is a distribution on standard tropical
polynomials, with expected degree $D < \infty$ and coefficient
differences bounded by some absolute constant.
Then $\G^k(\mathcal{F}, \mathcal{P})$ is an iterated Gilbert mosaic. 
\end{lem}
\begin{proof}
It is straight-forward to check that $\G^k(\mathcal{F}, \mathcal{P})$ 
satisfies the assumptions listed in Proposition \ref{prop:gk.mosaic}.
\end{proof}

It follows from Theorem \ref{thm:main2} that the sequence of mosaics
$\{ \G^k\left(\frac{1}{k}\right), k = 1, 2, \ldots\}$
has a limit in probability, and the rescaled sequence of common zeros
$\{\sqrt{k}\cdot \mathcal{I}_k, k = 1, 2, \ldots\}$
converges in probability to the process of intersections of $\G^\infty$.

\begin{thm}\label{thm:main.tropical}
Let $\mathcal{F}$ be a distribution on standard tropical polynomials, with finite expected degree and coefficient differences bounded by some absolute constant. 
Let $D_-, D_|, D_/$ be the expected number of arms in the directions spanned by vectors $(0,1)$, $(1,0)$ and $(1,1)$, respectively.
Let $\mathcal{P}$ be a homogeneous Poisson point process with intensity $\lambda$ on $\R^2$. 
Let $\mathcal{G}^k(\mathcal{F}, \mathcal{P})$ be the $k$-th order tropical plane curves mosaic. Let $\delta_-$, $\delta_/$ and $\delta_|$ be the Dirac delta measures at the points $0$, $\pi/4$ and $\pi/2$ on the unit circle, respectively. Let $W$ be a compact set in $\R^2$. As $k \to \infty$,
$$ \G^k\left(\frac 1 {k}\right)\cap W\stackrel{P}{\to} \G^\infty \cap W,$$
where $\G^\infty$ is the classical Poisson line process with cylindrical measure $\lambda \times (D_-\mu_- \delta_- + D_|\mu_| \delta_| + D_/\mu_/ \delta_/)$, 
where $\mu_- = \mu_| = \frac{2^{3/4}}{\sqrt{1 + \sqrt{2}}}$, and $\mu_/ = \left(\frac{\sqrt{2} + 3}{4}\right)\mu_-$.
\end{thm}

Let $D$ be the expected degree of a tropical polynomial distributed as $\mathcal{F}$. By Lemma \ref{lem:arms}, $D_-, D_/, D_| \leq D$, and thus $D < \infty$ implies that the three constants $D_-, D_/, D_|$ are finite also. Note that we view the $k$-th mosaic $\mathcal{G}^k$ as a random closed set, that is, we do not take into account the multiplicities of the arms. One cannot read off the multiplicity of an intersection in the limiting process, since doing so would have required the knowledge about the initial complex that the arms came from. However, one can still speak of the average multiplicity. By Lemma \ref{lem:arms}, the average multiplicities of arms in directions $(0,1)$, $(1,0)$ and $(1,1)$ are $D/D_-$, $D/D_|$ and $D/D_/$, respectively. By the Tropical B\'ezout's theorem \cite{maclagan2015introduction} , in expectation, the intersections of type $- \cap |$, $| \cap \ $ and $- \cap \ $ intensify with order $k\frac{D^2}{D_-D_|}$, $k\frac{D^2}{D_|D_\ }$ and $k\frac{D^2}{D_-D_
 \ }$, respectively. 

\begin{proof}
When the polynomials all have degree 1, $D_- = D_| = D_/ = 1$ by  Lemma \ref{lem:arms}, and this is an application Theorem \ref{thm:main} with $M = 3$, $A = \{0, \pi/4, \pi/2\}$, and at each site of the Poisson point process, there are precisely three motorcycles, one in each direction in $A$. For the general case, the following lemma gives a bound on the radius of the body of $V_f$ based on the pairwise differences in the coefficients of $f$. The result then follows from Theorem \ref{thm:main2}.
\end{proof}

\begin{lem}\label{lem:body.bound}
Let $f$ be a standard tropical polynomial of degree $d$ with coefficients $c_{ij}$. Suppose that the pairwise differences of the coefficients of $f$ are bounded by some constant $C$ independent of $d$, that is, 
$$ |c_{ij} - c_{kl}| \leq C, $$
for all coefficients $c_{ij}, c_{kl} < \infty$ of $f$. Then there exists a constant $R(C)$ such that the radius of the smallest ball containing the body of $V_f$ is at most $R(C)$.
\end{lem}
\begin{proof}
We shall prove that the set of vertices of $V_f$ is contained in the triangle with defining inequalities
\begin{align}
x & \leq C \label{eqn:body.bound.x} \\
y & \leq C \label{eqn:body.bound.y} \\
x - y & \geq -C. \label{eqn:body.bound.xy}
\end{align}
This would show that all line segments of the body of $V_f$ are also contained in this set, and thus proves the claim. 
Recall (cf. Section \ref{subsec:plane.curves}) that line segments and the arms of the tropical curves are normal to the edges of the Newton polygon of $f$. Consider $v_1 = (0,1)$, and let $(x,y)$ be a vertex of the body of $V_f$ supported by the hyperplane orthogonal to $v_1$. Then $(x,y)$ is dual to a cell of $Newt(f)$ that contains an edge of the form $((0,i_1), (0,i_2))$, for $0 \leq i_1 < i_2 \leq d$, $i_1, i_2 \in \mathbb{N}$. Thus,
$$ c_{i_10} + i_1 x = c_{i_20} + i_2 x, \Rightarrow x = \frac{c_{i_10} - c_{i_20}}{i_2 - i_1}. $$
Since $i_2 - i_1 \geq 1$,
$$ |x| = \frac{|c_{i_10} - c_{i_20}|}{i_2 - i_1} \leq C. $$
Thus, all points in the body of $V_f$ satisfy (\ref{eqn:body.bound.x}). A similar argument proves (\ref{eqn:body.bound.y}) and (\ref{eqn:body.bound.xy}).
\end{proof}

\subsubsection{Arm-body common zeros}
The scaling results obtained on ${\mathcal G}_k$ also
allow one to derive expressions for the asymptotic 
properties of the mean number of arm-body 
zeros of order $k$ per body.

The reference measure is now the Palm probability of $\mathcal P$,
which according to Slinyak's theorem, is the distribution
of $\mathcal P$ considered above, with an extra point added at the origin.
Equivalently, under the Palm setting, to the translation invariant
set of plane curves considered above,
one adds an independent plane curve centered at the origin. 
Condition on the fact that 
the body at the origin (or equivalently the typical body)
has a total segment length $l_-$, $l_/$, $l_|$, and $l_{o^1},\ldots,l_{o^i}$,
with the orientations $0$, $\pi/4$, $\pi/2$ and any other 
orientations $o^1,\ldots,o^i$ respectively.
For instance, for the example of Figure \ref{fig:duality},
there are two such directions, $o^1= -\frac \pi 4$ and
$o^2=\arccos(\frac 2{\sqrt{5}})$.
Then, when $\mathcal P$ has intensity $\frac {\lambda} k$,
the mean number of intersections of arms of order $k$ converges to
\begin{eqnarray}
M & = & 
\lambda 
\left(
l_- \left(D_| \mu_| + \frac 1 {\sqrt{2}} D_/\mu_/ \right)
+l_/\left(\frac 1 {\sqrt{2}} D_-\mu_- +\frac 1 {\sqrt{2}} D_|\mu_|\right) 
+l_|\left(D_-\mu_- +\frac 1 {\sqrt{2}} D_/\mu_/\right)
\right)\nonumber \\
& & + \lambda \sum_{j=1}^i l_{o^j}
\left(
\left|\sin(o^j - \frac \pi 2)\right|
D_| \mu_| 
+ 
\left|\sin(o^j)\right|
D_- \mu_- 
+
\left|\sin(o^j - \frac \pi 4)\right|
D_/\mu_/ \right)
\end{eqnarray}
when $k$ tends to infinity.
This formula follows from two results. The first
one is Theorem \ref{thm:main.tropical},
which, together with Slivnyak's theorem, implies that the process of
arms of order $k$ that cross the finite observation 
window containing the body in question converge
to a Poisson line process with the characteristics 
given in Theorem \ref{thm:main.tropical}. 
The second is the classical formula for the mean number of intersection
that a segment with a given orientation has with a translation
invariant (and non necessarily isotropic) Poisson line process.

The final formula is obtained by unconditioning with respect to the
distribution of the typical curve body.

\subsection{Stationary point processes of
tropical curves: lack of existence.}
We now give the justification for studying the tropical plane curves
process via the iterated Gilbert model, by showing that other
`natural' models fail to exist.
Any mechanism for generating the coefficients $c_{ij}$'s in (\ref{eqn:f})
randomly defines a distribution on tropical plane curves,
which is a distribution on random closed sets in $\R^2$.
Consider the goal of defining an appropriate 
`stationary collection of standard tropical curves',
whose set of common zeros (pairwise intersections) forms a 
stationary point process in $\R^2$. Here stationarity means invariance in law by classical translations, which translates to invariance in law by tropical scalings. In other words, we want to define a family of tropical polynomials whose set of roots form a tropical fractal: a set whose law is invariant under tropical scalings. 

By analogy with what exists for classical line processes, there are two natural ways to represent a random collection of tropical curves. The first is to view random tropical curves as random points in $\mathcal{C}$, the set of closed sets on $\R^2$. The second is to consider the union of a collection of tropical curves with stationary centroids as one large random polyhedral complex, and study its properties from the viewpoint of random mosaics of $\R^2$. Unfortunately, the following propositions state that neither such objects exist. 

\begin{prop}\label{prop:no.counting.process}
There exists no stationary, non-degenerate point processes of tropical curves with positive intensity. 
\end{prop}

\begin{prop}\label{prop:no.union}
There exists no set of tropical curves such that the following properties are simultaneously satisfied:
\begin{itemize}
  \item Its set of centroids form a stationary point process with positive intensity in $\R^2$.
  \item Its set of common zeros is the support of a point process in $\R^2$.
\end{itemize}
\end{prop}
The appendix contains a review of stochastic geometry and proof of the above propositions. 

\subsection{The tropical lines process}
\begin{defn}[Tropical lines process]\label{def:trop.line.process}
Let $\mathcal{F}$ be the atomic measure on the tropical lines centered at the origin, that is, the tropical line
$$ V_f := \{x = 0, y \geq 0\} \cup \{x \geq 0, y = 0\} \cup \{x = y \leq 0 \}$$
corresponds to the polynomial $f$ of degree one
$$ f: \R^2 \to \R, f(x,y) = x \oplus y \oplus 0. $$
The $k$-th order iterated Gilbert model $\mathcal{G}^k(\mathcal{P}, \mathcal{A})$, denoted $\mathcal{G}^k_L$, is called the $k$-th order tropical lines process of intensity $\lambda$. 
\end{defn}

\begin{figure}[h]
\includegraphics[width = 0.8\textwidth]{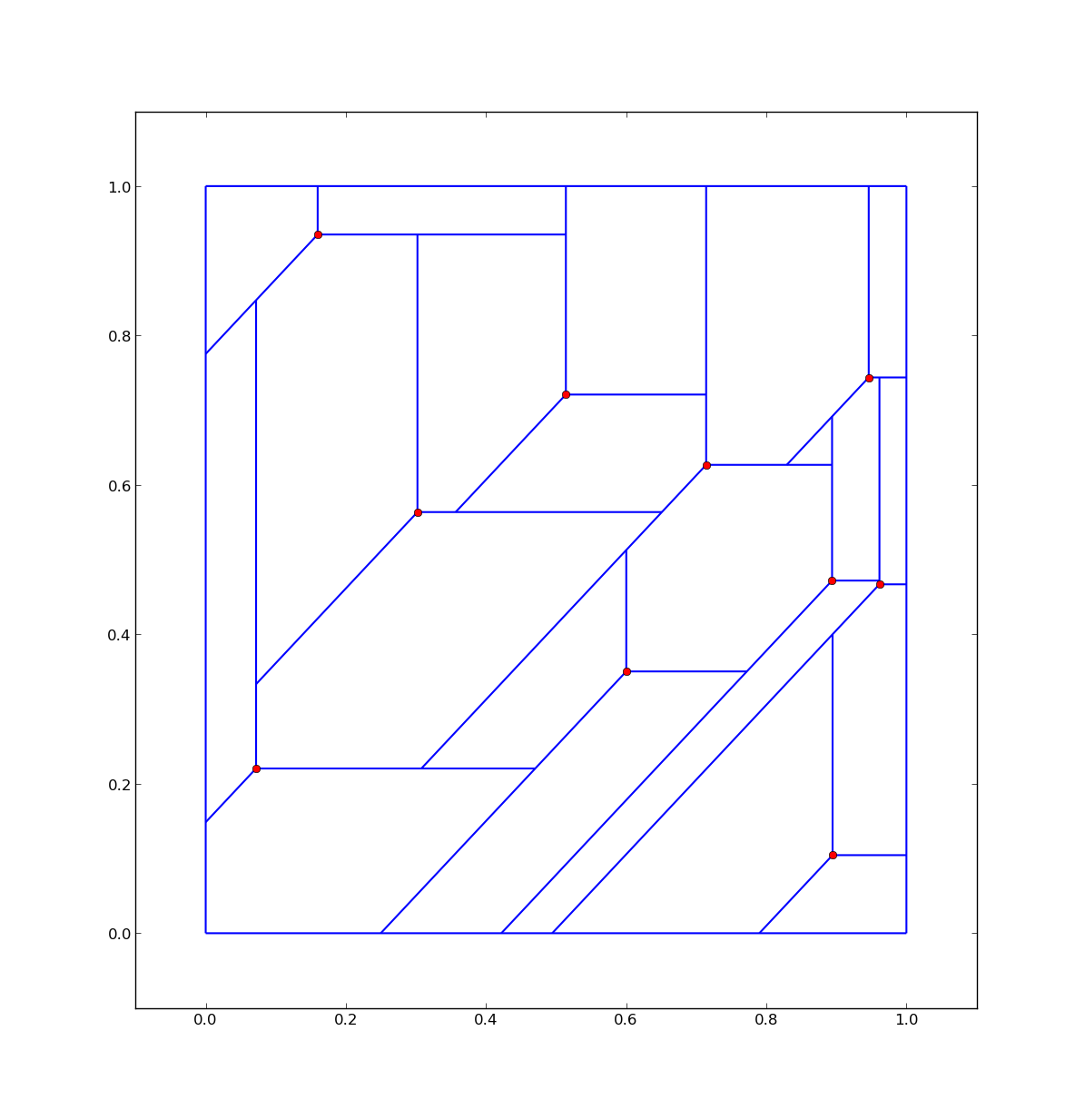}
\caption{A simulation of the first order tropical lines process $\mathcal{G}^1_L$ with $\lambda = 10$, restricted to the square $[0,1]^2$. The red points are the starting positions of the motorcycles, which form a Poisson point process in $\R^2$ with rate $\lambda$. The blue lines are the motorcycles' trails.}
\end{figure}

The tropical lines process can be defined directly as an iterated Gilbert model as follows. Let $m = 3$, $A$ be the set of angles $\{0, \pi/2, 5\pi/4\}$ relative to $(1,0)$, called east, north and southwest directions, respectively. Let the measure $\mathcal{A}_2$ has mass 0, and $\mathcal{A}_3$ has mass 1 on the set $A$. In other words, at each point of the Poisson point process $\mathcal{P}$ with rate $\lambda > 0$, put three motorcycles, one in each direction in $A$. Then $\G^k_L = \G^k(\mathcal{P}, A)$. The exact and asymptotic intensities of the vertice, edge and face processes for the $k$-th order tropical line process $\G^k_L$ follow from Proposition \ref{prop:gk-finite}. 

\begin{cor}\label{cor:intensity}
For $k \geq 1$, let $\lambda_0^k, \lambda_1^k, \lambda_2^k$, be the intensities of the vertices, edges and cells of $\G^k_L$,  respectively. Then
\begin{align*}
\lambda_0^k &= (3k+1)\lambda, \\ 
\lambda_1^k &= 6k\lambda,  \\
\lambda_2^k &= (3k-1)\lambda.
\end{align*}
In particular, the vertex, edge and face intensities of the scaled limit $\G^\infty_L$ are 
$$\lambda_0^\infty = 3\lambda, \hspace{0.5em} \lambda_1^\infty = 2\lambda, \hspace{0.5em} \lambda_2^\infty = 3\lambda.$$
\end{cor}

\begin{cor}
Let $p_{- \cap |}$, $p_{- \cap /}$ and $p_{| \cap /}$ denote the intensities of the three types of intersections of $\G^\infty_L$. Then 
\begin{equation}
p_{- \cap |} = \lambda\frac{2\sqrt{2}}{\sqrt{2} + 1} , \hspace{0.5em} p_{| \cap /} = p_{- \cap /} = \lambda\left(\frac{\sqrt{2} + 3}{2}\right)\frac{1}{\sqrt{2} + 1}.
\end{equation}
\end{cor}
Note that the intensity of intersections, $\lambda^\infty_0$, is
$$ \lambda^\infty_0 = \lambda\left(\frac{\sqrt{2} + 3}{\sqrt{2} + 1} + \frac{2\sqrt{2}}{\sqrt{2} + 1}\right) = 3\lambda, $$
which agrees with Corollary \ref{cor:intensity}.
\begin{proof}
By Theorem \ref{thm:main.tropical}, $\G^\infty_L$ is the Poisson line process consisting of three types of lines with directions $(1,0)$, $(0,1)$ and $\frac{1}{\sqrt{2}}(-1,-1)$, with intensities $\mu_-, \mu_|$ and $\mu_/$. We can fix different observation windows $W$ to calculate the intensity for each type of intersection. Let $\lambda = 1$. For $- \cap |$, let $W$ be a square of sidelength 1. Thus,
$$p_{- \cap |} = \mu_- \mu_| = \frac{2\sqrt{2}}{\sqrt{2} + 1}.$$ 
For $- \cap /$, let $W$ be the parallelogram formed by the vectors $(1,0)$ and $(-1,-1)$. There are $\mu_-$ many horizontal lines, and $\frac{1}{\sqrt{2}}\mu_/$ many diagonal lines crossing $W$. Thus 
$$ p_{- \cap /} = \frac{1}{\sqrt{2}}\mu_- \mu_/ = \left(\frac{\sqrt{2} + 3}{2}\right)\frac{1}{\sqrt{2} + 1}. $$
The intensity of $| \cap /$ equals that of $- \cap /$ by symmetry. 
\end{proof}

Now consider the faces of $\G^k_L$.
The edge directions of each face are one of the three directions of $A$. One can check that each face is an ordinary and tropically convex set, that is, it is a polytrope \cite{joswig2010tropical}. Classifying polytropes by their combinatorial type is an interesting problem \cite{joswig2010tropical,tran2013enumerating}. Let us compute the intensities of various polytropes by their combinatorial types in $\G^k_L$. 

Polytropes in $\R^2$ have $3, 4, 5$ or $6$ proper vertices. For $i = 0, \ldots, 3$, let $p^k_i$ denote the intensity of cells of $\G^k_L$ with $3+i$ vertices. Then 
\begin{equation}\label{eqn:facetype1}
p^k_0 + p^k_1 + p^k_2 + p^k_3 = \lambda_2^k = (3k-1)\lambda.
\end{equation}
Now, each point in $\Psi$ is the proper vertex of three faces, each intersection of order $3$ in $\G^k_L$ is the proper vertex of two faces, and each intersection of order $4$ in $\G^k_L$ is the proper vertex of four faces. Thus,
\begin{equation}\label{eqn:facetype2}
3p^k_0 + 4p^k_1 + 5p^k_2 + 6p^k_3 = 3\lambda + 2 \cdot 3\lambda + 4 \cdot 3(k-1) \lambda = (12k-3)\lambda.
\end{equation}

For $p_i$ the intensity of cells of $\G^\infty_L$ with $3+i$ vertices for $i = 0, 1, 2, 3$, we have
\begin{align}
p_0 + p_1 + p_2 + p_3 &= 3\lambda \label{eqn:facetype1.infty} \\
3p_0 + 4p_1 + 5p_2 + 6p_3 &= 12\lambda, \label{eqn:facetype2.infty}
\end{align}
While (\ref{eqn:facetype1.infty}) and \ref{eqn:facetype2.infty}) do not uniquely determine the intensities $p_i$, we can compute these numbers directly from the Poissonian description of $\G^\infty_L$.

\begin{lem}
Let $p_i$ be the intensity of cells of $\G^\infty_L$ with $3+i$ vertices for $i = 0, 1, 2, 3$, we have
\begin{align*}
p_3 &\approx 0.429367312053161, & p_4 &\approx 2.22221756362048, \\
p_5 &\approx 0.267462936599565, \hspace{1em} \mbox{ and }  & p_6 &\approx 0.0809521877267980.
\end{align*}
\end{lem}

\begin{proof}
Fix a rectangle of width $x$, height $y$, with $x > y$. Let $N_a$, $N_b$ and $N_c$ be the number of diagonal lines crossing regions $A, B, C$ in Figure \ref{fig:abc} below. Note that these are independent Poisson random variables, with means $a = \min(x,y)/\sqrt{2}$, $b = |y-x|/\sqrt(2)$ and $c = a$, respectively. Conditioned on the values of $N_a$, $N_b$ and $N_c$, we can count the number of triangles, quadilaterals, pentagons and hexagons generated. See Table \ref{tab:abc}. 

\begin{figure}[h]
\includegraphics{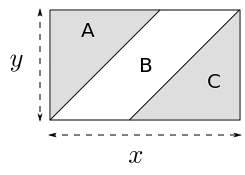}
\caption{An $x \times y$ rectangle with $x > y$ is divided into three regions labelled $A,B$ and $C$ by two diagonal lines. Conditioned on the number of diagonal lines of $\G^\infty_L$ hitting each region, one can compute the intensities of faces of $\G^\infty_L$ by vertices.}
\label{fig:abc}
\end{figure}

\begin{table}[h]
\begin{tabular}{| c | c | c | c | c | c | c |}
\hline
$N_a$ & $N_b$ & $N_c$ & $\triangle$ & $\square$ & \pentagon & \varhexagon \\
\hline
0 & 0 & 0 & 0 & 1 & 0 & 0 \\
1 & 0 & 0 & 1 & 0 & 1 & 0 \\
0 & 0 & 1 & 1 & 0 & 1 & 0 \\
$\geq 1$ & 0 & $\geq 1$ & 2 & $N_a + N_c - 2$ & 0 & 1 \\
0 & $\geq 1$ & 0 & 0 & $N_b + 1$ & 0 & 0 \\
1 & $\geq 1$ & 0 & 1 & $N_b$ & 0 & 0 \\
0 & $\geq 1$ & 1 & 1 & $N_b$ & 0 & 0 \\
$\geq 1$ & $\geq 1$ & $\geq 1$ & 2 & $N_a + N_b + N_c - 3$ & 2 & 0 \\
\hline
\end{tabular}
\vskip12pt
\caption{Number of polytropes of each type generated by a single rectangle, conditioned on the number of diagonal lines intersecting it.} \label{tab:abc}
\end{table}

Let $i$ be an index variable taking values in $\{3,4,5,6\}$.
Using the previous table, for each fixed rectangle of size $x \times y$, we can compute $e_i(x,y)$, the expected number of polytropes with $i$ vertices. Now fix a large square $W$ of side length $s$. Consider the Manhattan line process with horizontal and vertical intensities $\mu_-$. There are $s^2 \mu_-^2 + o(s^2)$ many rectangles in $W$. The side lengths of the rectangles are distributed as i.i.d. exponential with mean $1/\mu_-$. Then 
$$ s^2\int_{x,y} e_i(x,y) e^{-\mu_-(x+y)}\, dx \, dy + o(s^2) $$
is the expected number of polytropes with $i$ vertices in $W$. Thus, as $s \to \infty$, the density of polytropes with $i$ vertices is precisely
$$ p_i = \int_{x,y} e_i(x,y) e^{-\mu_-(x+y)}\, dx \, dy$$
for $i \in \{3,4,5,6\}$. Numerically evaluate these integrals yield the result. 
\end{proof}

\section{Discussions} \label{sec:discussions}

Our work leaves a number of open questions for both algebraic and stochastic geometers.
To be concrete, we list a few such problems:

\begin{enumerate}
  \item Is there a distributional limit for the error term 
$d\left(\G^k\left(\frac{1}{k}\right) \cap W , \G^\infty\cap W\right)$,
where $d$ is some distance between sets?
For a classical Poisson line process,
the fluctuations of such statistics around the limit
are Gaussian. In this case, we have extra fluctuations
from the Gilbert iterations. We suspect that the fluctuations
are still Gaussian, but with higher variance.
\item What happens when the number of angles $A$ is not finite?
For example, what is the limit of the iterated classical Gilbert
tessellation with uniform angle distribution on $[0,\pi]$?
\item What happens when the underlying point process $\mathcal P$
is stationary but not Poisson? In particular, for what types
of point processes beyon Poisson are the dilation scaling results obtained
here still valid?
\item What happens to systems of tropical polynomials in higher variables? The intersections would now be segments of hyperplanes of various codimensions. These processes will also be dense, and thus one needs a way to enumerate the common zeros. However, it is not clear what is the effective analogue of the iterated Gilbert tessellation in higher dimensions.
  \item Is the Poisson tropical plane curve process the tropicalization of some processes of classical varieties? Tropicalization is often studied in the field of Puiseux series, or the $p$-adics. There has been work by Evans \cite{evans2006expected} on systems of polynomials whose coefficients are $p$-adic Gaussians. However, their tropicalizations would result in discretely distributed coefficients, and thus the current result does not directly apply.
\end{enumerate}

\bibliographystyle{plain}
\bibliography{iteratedGilbert}

\appendix
\section*{Appendix}
In this section, we first review some basics of stochastic geometry, give the precise definitions of the terms and the proof of these propositions. 

\subsubsection*{Background}
We first review some terminologies in stochastic geometry, for reference see \cite{schneider2008stochastic} and \cite{daley2007introduction}. For a set $\{\cdot\}$, let $\text{Card}\{\cdot\}$ denote its cardinality. For $V\subset \R^2$ and $x \in \R^2$, let $x + V$ be the translated set $\{x + y: y \in V\}$. Let $\mathcal{C}$ denote the set of closed sets in $\R^2$. 
Let $\mathcal{C}^1 \subset \mathcal{C}$ be the set of tropical curves with finite and positive degrees. Note that when we write $V_f \in \mathcal{C}^1$, we view the tropical curve $V_f$ as a closed set in $\R^2$. In contrast, when we write $V_f \in \mathcal{V}$, we view it as a compact set with marked points. For notational convenience, we will suppress the dependence on $f$ and write $V_n$ for $V_{f_n}$. 
Let $\mathcal{B}$ denote the Borel $\sigma$-algebra of the Fell topology on $\mathcal{C}$. 
A point process on $\mathcal{C}^1$ is a counting measure of the form
\begin{equation}
\label{eq:phi}
\Phi= \Phi(\omega)= \sum_{n \in \mathbb{N}} \delta_{V_n},
\end{equation}
that satisfies the following $\sigma$-finiteness (or Radon) condition: for all compact sets $K \subset \R^2$,
\begin{equation}
\label{eq:choquet}
\mbox{Card}\{n \mbox{ s.t. } V_n\cap K \ne \emptyset\}<\infty, \quad
\P  \mbox{-a.s.}.
\end{equation}
Here $V_n$ are random variables taking values in $\mathcal{C}^1$, that is, they are random tropical curves.
The intensity measure $\Lambda$ of $\Phi$ is the measure on $\mathcal{B}$ given by
$$ \Lambda(A) = \E(\Phi(A)) \mbox{ for } A \in \mathcal{B}. $$
Say that $\Phi$ is a \term{non-degenerate point process of tropical plane curves} if $\Lambda$ is supported on $\mathcal{C}^1$. Say that $\Phi$ is \term{stationary} if its distribution is invariant under actions by the group of translations of $\R^2$. Say that $\Phi$ has \term{positive intensity} if
for all compact sets $K$ of $\R^2$ with positive Lebesgue measure,
$$ \E\left(\mbox{Card} \{n \mbox{ s.t. } V_n\cap K \ne \emptyset\}\right) >0.$$
If $\Phi$ satisfies the last three properties, that is, it is a non-degenerate, stationary point process of tropical plane curves with positive intensity, then we say that $\Phi$ is a \emph{stationary point process of tropical plane curves}. 

For a centroid function $c$ on tropical curves, one can associate with $\Phi$ the point process of centroids
\begin{equation}\label{eqn:psi}
\Psi = \sum_{n \in \mathbb{N}} \delta_{C_n}, 
\end{equation}
where each $C_n$ is a random variable in $\R^2$, $C_n(\omega) = c(V_n(\omega))$, the centroid of the tropical plane curve $V_n(\omega)$. 
We are now ready to prove Proposition \ref{prop:no.counting.process}. 

\begin{proof}[Proof of Proposition \ref{prop:no.counting.process}.]
The proof is by contradiction. Assume there exists
$$ \Phi= \Phi(\omega)= \sum_n \delta_{V_n},$$
which is Radon and stationary in the sense defined above. Consider the map $c: \mathcal{C}^1 \to \R^2$, where $c(V_n)$ is the apex of the horizontal arm of $V_n$ whose apex has minimum $y$-coordinate. This function is well-defined, measurable and translation invariant, therefore it is a centroid function on $\mathcal{C}^1$. Let $\Psi$ be the centroid point process of $\Phi$ with centroid function $c$. Since $\Phi$ is stationary and has positive intensity, and each curve in $\Phi$ yields precisely one point in $\Psi$, $\Psi$ is a stationary point process in $\R^2$ with positive intensity.

We now claim that
\begin{equation}
\label{eq:infty}
\P ( \Psi ((-\infty,0]\times [0,1]) = \infty)>0.
\end{equation}
The projection on the abscissa axis
of the restriction of $\Psi$ to the set $(-\infty,\infty)\times [0,1])$
forms a stationary point process on $\R$.
If (\ref{eq:infty}) is not true, then it follows from the last observation 
and from Property 1.1.2  in \cite{baccelli2013elements} that
$$\P ( \Psi ((-\infty,\infty)\times [0,1]) = 0)=1,$$
which in turns implies that
$$\P ( \Psi ((-\infty,\infty)\times [n,n+1]) = 0)=1,$$
for all $n\in \Z$, so that
$$\P ( \Psi (\R^2)=0)=1,$$
and this contradicts the fact that the intensity of $\Psi$ is positive. So (\ref{eq:infty}) must holds.

Now, each tropical curve with centroid in $(-\infty, 0] \times [0, 1])$ has at least one arm extending in the $(1,0)$ direction starting from this centroid point. Therefore, this arm intersects the $[(0,0), (0,1)]$ segment of $\R^2$. Hence, with positive probability, there is an infinite number of different horizontal arms of tropical lines of $\Phi$ that intersect $[(0,0), (0,1)]$. This contradicts (\ref{eq:choquet}). 
\end{proof}

\subsubsection*{The second view}
Alternatively, one could start with a stationary point process $\Psi$ in $\R^2$ as given in (\ref{eqn:psi}), and to each point $C_n$ attach a random tropical plane curve, translated to have $C_n$ as the centroid. One could then consider the set 
\begin{eqnarray}
\label{eqn:union}
\Xi = \Xi(\omega)= \bigcup_n V_n(\omega),
\end{eqnarray}
where $V_n(\omega)$ has centroid $C_n(\omega)$. If it exists, $\Xi$ would be a random variable taking values in $\mathcal{C}$. 
As the following example shows, it is possible to have $\Xi$
to be a well defined stationary random closed set, whereas as we know from the last lemma that the associated process $\Phi = \sum_n \delta_{V_n}$ is not a point process on $\mathcal{C}$.\\

\noindent
\begin{ex}\label{ex1} 
Choose the centroid function $c$ as in the proof of Proposition \ref{prop:no.counting.process}. The centroid of a tropical line of the form $a \odot x \oplus b \odot y \oplus c$ is just the only point on its body, which is $(c-a,c-b)$. In (\ref{eqn:union}), take
$n=(k,l)$ varying over $\Z^2$, and for
$T_{k,l}$ with $k,l\in \Z$, take the tropical line
with centroid $(k,l)+U$, where $U$ is a random variable
which is uniformly distributed in the cube $[0,1]\times[0,1]$.
The associated $\Xi$ is depicted in Figure \ref{fig1}.

\begin{figure}[h]
 \begin{center}
  \includegraphics[width=4cm]{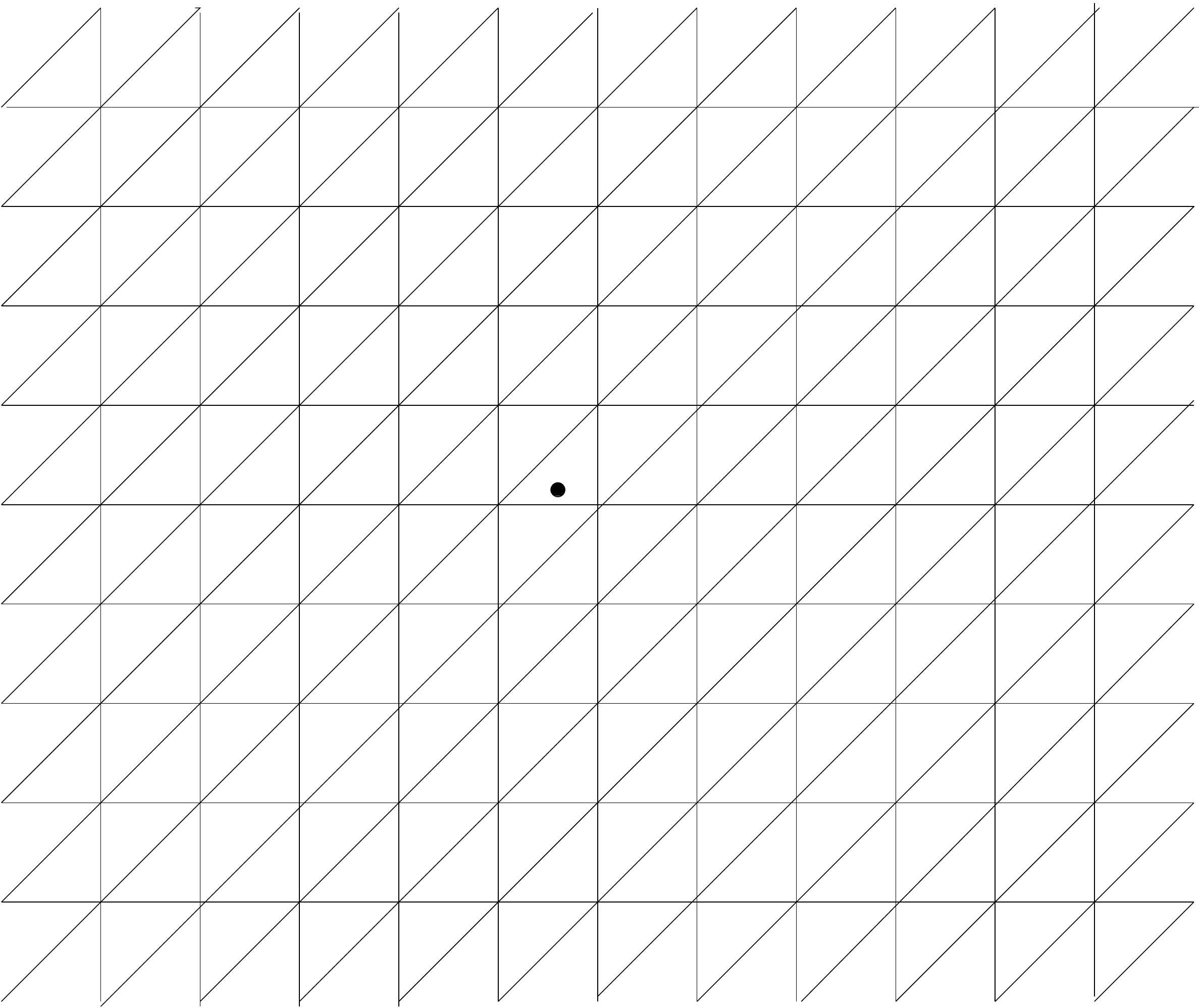}
 \end{center}
 \caption{An instance of $\Xi$}
\label{fig1}
\end{figure}

There is no contradiction with Proposition \ref{prop:no.counting.process}.
A direct evaluation shows that
the centroid point process $\Psi$ has intensity 1, and that
with probability 1,
the number of tropical lines that intersect the 
$[(0,0),(0,1)]$ segment of $\R^2$ is infinite.
However, all vertical (resp. horizontal or diagonal)
arms that intersect this segment do so at the same point.
This is of course directly linked to this specificity of this
example where
each tropical line has a degenerate intersection with 
an infinite number of other tropical lines of $\Xi$. 
\end{ex}

To rule out pathological cases as that in Example \ref{ex1}, we shall consider tropical curves with non-degenerate intersections. Say that two tropical curves $V, V' $ have non-degenerate intersections, if $V \cap V'$ is a set of finitely many points in $\R^2$. In algebraic terms, this means the corresponding system of two polynomials in two variables is not singular. For a collection $\{V\}$ of tropical curves, the union of their pairwise intersections is called the \term{set of common zeros of $\{V\}$}.
If $\{V\}$ is a $\P$-a.s. non-degenerate collection of tropical curves, the set of common zeros is a point process in $\R^2$. We call this the \term{intersection process} of $\{V\}$, denoted by $\mathcal{I}$. 

In Example \ref{ex1}, the set of common zeros of the collection of lines of $\Xi$  is the random closed set $\Xi$ itself, which is not the support of a point process on $\R^2$, as it contains lines. Unfortunately, as claimed in Proposition \ref{prop:no.union}, we cannot define $\Xi$ such that both the associated centroid process and the intersection process are stationary with positive intensities. 

\begin{proof}[Proof of Proposition \ref{prop:no.union}]
Consider the maps $c_1, c_2: \mathcal{C}^1 \to \R^2$, where $c_1(V)$ is the apex of the horizontal arm of $T$ with minimum $y$-coordinate, and $c_2(V)$ is the apex of the vertical arm of $T$ with minimum $x$-coordinate. As argued in the proof of Proposition \ref{prop:no.counting.process}, both of these are centroid functions of tropical curves. 

Consider the set $\Xi$ defined in (\ref{eqn:union}) via some stationary, positive intensity centroid point process $\Psi$. Let $\Psi_1$ be the centroid process of $\Xi$ associated with the function $c_1$, $\Psi_2$ be the centroid process of $\Xi$ associated with the function $c_2$. There is a bijection between points of $\Psi$, $\Psi_1$ and $\Psi_2$ via the tropical curves $T_n$. Since centroid functions are translation invariant, and since $\Psi$ is a stationary point process with positive intensity, it follows that $\Psi_1$ and $\Psi_2$ are also.

Now, either 
$$\P(\exists m\ne n \mbox{ s.t. } V_m  \mbox{ and } V_n  \mbox{ have  degenerate intersection})>0$$
in which case the set of common zeros contains a half line with
positive probability, and the result follows.
Or
$$\P(\exists m\ne n \mbox{ s.t. } V_m  \mbox{ and } V_n  \mbox{ have  
degenerate intersection})=0.$$
Consider this second case. Let 
$$\tau_1 = \{V_n: c_1(V_n) \in (-\infty,0] \times [0,1]\},$$ 
and
$$ \tau_2 = \{V_n: c_2(V_n) \in [-1,0]\times (-\infty,0]\}. $$
Apply the argument in Proposition \ref{prop:no.counting.process} to $\Psi_1$, we get
\begin{equation}\label{eqn:tau1}
\P(\mbox{Card}(\tau_1) = \infty) > 0. 
\end{equation}
Since $\Psi_2$ has positive intensity, 
\begin{equation}\label{eqn:tau2}
\P(\mbox{Card}(\tau_2) > 0) = 1.
\end{equation}

Each $V_n = V_{f_n}$ is dual to the subdivision of the triangle with vertices $(0,0), (0,d), (d,0)$, where $d$ is the degree of $f_n$. Therefore, by definition of $c_1$ and $c_2$,
the $y$-coordinate of $c_1(V_n)$ is at most the $y$-coordinate of $c_2(V_n)$, and the $x$-coordinate of $c_2(V_n)$ is at most the $x$-coordinate of $c_1(V_n)$. Therefore, it is not possible for a tropical curve to be in both $\tau_1$ and $\tau_2$. 

Let 
$$ I = \bigcup_{V_n \in \tau_1, V_m \in \tau_2} V_n \cap V_m $$
be the subset of the set of common zeros which are intersections of pairs of tropical lines in $\tau_1$ and $\tau_2$. Since such intersections are a.s. non-degenerate, $I$ must be a collection of points in $[-1,0] \times [0,1]$. By definition of $c_1$, for each point in $c_1$, there is a tropical curve $V_n$ with a horizontal arm ($(1,0)$ direction) starting from this point. Similarly, for each point in $c_2$, there is a tropical curve $V_m$ with a vertical arm ($(0,1)$ direction) starting from this point. Thus, each set $V_n \cap T_m$ necessarily contains a point $p_{mn} \in [-1,0] \times [0,1]$ that is the intersection of this horizontal and vertical arms.

By the non-degenerate intersection assumption, $\P$-a.s. there are no two tropical plane curves with identical $c_1$-centroids or $c_2$-centroids. Therefore,$p_{mn} \neq p_{m'n'}$ whenever either $m \neq m'$ or $n \neq n'$. That is, all the intersections from different pairs are different. 

Thus, by (\ref{eqn:tau1}) and (\ref{eqn:tau2}), 
$$ \P(\mbox{Card}(I) = \infty) > 0. $$

Hence, the set of common zeros of $\Xi$ cannot be Radon in this case too.
\end{proof}

\end{document}

%% file: preamble.tex
\usepackage[body={6.3in,9.3in}]{geometry}
\usepackage{verbatim}
\usepackage[ansinew]{inputenc}

\usepackage{textcomp}
\usepackage{latexsym}
\usepackage[usenames]{color}
\usepackage{amsmath,amsfonts,amssymb,amsthm}
\usepackage{graphicx}
\usepackage{longtable}
\usepackage{bbm, cite}
\usepackage{algorithm, algorithmicx, algpseudocode}

\usepackage{setspace}
\usepackage{array}

\def\qed{\hfill{\raggedleft{\hbox{$\Box$}}} \smallskip}

\def\R{\mathbb{R}}

\def\A{\mathcal{A}}

\def\E{\mathbb{E}}
\def\P{\mathbb{P}}

\newcommand{\vect}[1]{\overrightarrow{#1}}

\theoremstyle{plain} \newtheorem{lem}{Lemma}
\theoremstyle{plain} \newtheorem{prop}[lem]{Proposition}
\theoremstyle{plain} \newtheorem{thm}[lem]{Theorem}
\theoremstyle{plain} \newtheorem{cor}[lem]{Corollary}
\theoremstyle{plain} 
\theoremstyle{plain} 
\theoremstyle{definition} \newtheorem{defn}[lem]{Definition}
\theoremstyle{definition}
\theoremstyle{definition} 
\theoremstyle{definition} 
\theoremstyle{definition}\newtheorem{ex}[lem]{Example}

\newlength\savedwidth

%% file: angles.pstex_t
\begin{picture}(0,0)%
\includegraphics{angles}%
\end{picture}%
\setlength{\unitlength}{4144sp}%
\begingroup\makeatletter\ifx\SetFigFont\undefined%
\gdef\SetFigFont#1#2#3#4#5{%
  \reset@font\fontsize{#1}{#2pt}%
  \fontfamily{#3}\fontseries{#4}\fontshape{#5}%
  \selectfont}%
\fi\endgroup%
\begin{picture}(15476,7226)(-1994,-9090)
\put(11566,-2131){\makebox(0,0)[lb]{\smash{{\SetFigFont{20}{24.0}{\rmdefault}{\mddefault}{\updefault}{\color[rgb]{0,0,0}$T_{\phi,\varphi_3}$}%
}}}}
\put(5446,-7981){\makebox(0,0)[lb]{\smash{{\SetFigFont{20}{24.0}{\rmdefault}{\mddefault}{\updefault}{\color[rgb]{0,0,0}$T_{\phi,\varphi_1}$}%
}}}}
\put(12511,-8296){\makebox(0,0)[lb]{\smash{{\SetFigFont{20}{24.0}{\rmdefault}{\mddefault}{\updefault}{\color[rgb]{0,0,0}$T_{\phi,\varphi_2}$}%
}}}}
\put(-1979,-4381){\makebox(0,0)[lb]{\smash{{\SetFigFont{20}{24.0}{\rmdefault}{\mddefault}{\updefault}{\color[rgb]{0,0,0}$w_{\varphi_1} \vec \varphi_1$}%
}}}}
\put(4231,-5686){\makebox(0,0)[lb]{\smash{{\SetFigFont{20}{24.0}{\rmdefault}{\mddefault}{\updefault}{\color[rgb]{0,0,0}$w_\phi \vec \phi$}%
}}}}
\put(4681,-2176){\makebox(0,0)[lb]{\smash{{\SetFigFont{20}{24.0}{\rmdefault}{\mddefault}{\updefault}{\color[rgb]{0,0,0}$w_{\varphi_2} \vec \varphi_2$}%
}}}}
\put(1261,-8521){\makebox(0,0)[lb]{\smash{{\SetFigFont{20}{24.0}{\rmdefault}{\mddefault}{\updefault}{\color[rgb]{0,0,0}$w_{\varphi_3} \vec \varphi_3$}%
}}}}
\put(6526,-6001){\makebox(0,0)[lb]{\smash{{\SetFigFont{20}{24.0}{\rmdefault}{\mddefault}{\updefault}{\color[rgb]{0,0,0}$b$}%
}}}}
\put(10711,-6046){\makebox(0,0)[lb]{\smash{{\SetFigFont{20}{24.0}{\rmdefault}{\mddefault}{\updefault}{\color[rgb]{0,0,0}$b+w_\phi \vec \phi$}%
}}}}
\end{picture}%